\documentclass[3p,preprint,onecolumn]{elsarticle}
\usepackage[utf8]{inputenc}
\usepackage{graphicx}
\graphicspath{{Figures/}}
\usepackage{dcolumn}
\usepackage{mathtools,amsthm}
\usepackage{bm}
\usepackage[numbers]{natbib}

\usepackage{moreverb}
\usepackage{color}
\usepackage{amssymb} 
\usepackage{multirow} 
\usepackage{subfig} 
\usepackage{stmaryrd} 
\usepackage{enumitem}
\setdescription{itemsep=0pt,parsep=0pt,leftmargin=0.5cm}

\usepackage{tikz}

  \newtheoremstyle{break}
   {\topsep}{\topsep}%
   {\itshape}{}%
   {\bfseries}{}%
   {\newline}
   {\thmname{#1}\thmnumber{\@ifnotempty{#1}{ }\@upn{#2}}%
    \thmnote{ {\bfseries(#3)}}}% 
\makeatother
\theoremstyle{break}

\theoremstyle{remark}
\newtheorem{remark}{Remark}
 
\newtheorem{theorem}{\textbf{Theorem}}[section]
\newtheorem{lemma}[theorem]{\textbf{Lemma}}

\newtheorem{assumption}[theorem]{\textbf{Assumption}}

\definecolor{pink}{rgb}{1,0,1}
\definecolor{changecolor}{rgb}{0,0,0}

\newcommand{\zetalow}{\delta_\zeta}
\renewcommand{\i}{\textrm i}
\newcommand{\jump}[1]{\llbracket#1\rrbracket}
\renewcommand{\Re}{\operatorname{Re}}
\renewcommand{\Im}{\operatorname{Im}}
\newcommand{\RR}{\mathbb R}

\let\bs\boldsymbol

\begin{document} 

\title{A Nitsche-type  Method for Helmholtz Equation with an Embedded Acoustically Permeable Interface}

  \author[add1]{Esubalewe Lakie Yedeg\corref{cor}}
  \ead{yedegl@cs.umu.se}
  \author[add1]{Eddie Wadbro}
  \ead{eddiew@cs.umu.se} 
  \author[add2]{Peter Hansbo}
  \ead{Peter.Hansbo@ju.se} 
  \author[add3]{Mats G. Larson}
  \ead{mats.larson@math.umu.se} 
  \author[add1]{Martin Berggren\corref{cor}}
  \ead{martin.berggren@cs.umu.se} 

  \cortext[cor]{Corresponding author}
  \address[add1]{Department of Computing Science, Umeå University, Sweden}
  \address[add2]{Department of Mechanical Engineering, Jönköping University, Sweden}
  \address[add3]{Department of Mathematics and Mathematical Statistics, Umeå University, Sweden}

\date{\today}
\begin{abstract}

We propose a new finite element method for Helmholtz equation in the situation where an acoustically permeable interface is embedded in the computational domain. 
A variant of Nitsche's method, different from the standard one, weakly enforces the impedance conditions for transmission through the interface. 
As opposed to a standard finite-element discretization of the problem, our method seamlessly handles a complex-valued impedance function $Z$ that is allowed to vanish. 
In the case of a vanishing impedance, the proposed method reduces to the classic Nitsche method to weakly enforce continuity over the interface.
We show stability of the method, in terms of a discrete Gårding inequality, for a quite general class of surface impedance functions, provided that possible surface waves are sufficiently resolved by the mesh.
Moreover, we prove an a priori error estimate under the assumption that the absolute value of the impedance is bounded away from zero almost everywhere.
Numerical experiments illustrate the performance of the method for a number of test cases in 2D and 3D with different interface conditions.
\end{abstract} 
%\pacs{43.50Gf, 43.20Mv, 43.55Ka, 43.20Bi}
%\end{frontmatter}
\maketitle

\section{Introduction}	

In the context of acoustic or electromagnetic wave propagation, material properties of domain boundaries or thin embedded interfaces are commonly characterized in terms of a \emph{surface impedance} $Z$.
For governing equations written in second-order form and in frequency domain, the surface impedance condition is straightforward to enforce weakly as a natural condition in the corresponding variational form.
The surface impedance then appears in the denominator of a boundary term in variational form.
The limit $Z\to0$ corresponds to a Dirichlet condition, so the case $Z=\epsilon$ for a small $\epsilon>0$ can  be considered as an approximate treatment of a Dirichlet condition.

This approximation corresponds to the penalty method championed by Babuška~\cite{Ba73} to impose Dirichlet boundary conditions in the context of finite element methods.
Viewed as a numerical implementation of the Dirichlet condition, this  penalty method is simple to use but suffers from the fact that it is not consistent with the equation for the exact condition, which mean that the method will not be optimal-order accurate in general.
This method may also yield ill-conditioned system matrices, particularly for higher order elements.
An improvement that addresses these issues was suggested by Nitsche~\cite{Ni71}, and his ideas have been the basis for a wide range of further developments. 
Interior-penalty discontinuous Galerkin methods~\cite{ArBrCoMa02} use the ideas of Nitsche to enforce inter-element continuity.
Nitsche’s approach can also be used for domain decomposition and as a mortar method for meshes that do not match node-wise across an interface~\cite{St98,Becker2003}.
Juntunen \& Stenberg~\cite{Juntunen2009} extended Nitsche's method, designed for pure Dirichlet conditions, to a general class of mixed boundary conditions.
Hansbo \& Hansbo~\cite{Hansbo20043523} introduced a Nitsche-type method for static linear elasticity in order to handle imperfect bounding, modeled with elastic spring-type conditions, across an embedded interface.
Recently, there has also been an intense development of so-called  cut finite element techniques, where interfaces, typically supporting jumps in the solution across the interface, are allowed to cut arbitrarily across a background mesh~\cite{BuClHaLaMa15}.
The transmission conditions at the interface are in these methods handled by variations of the idea by Nitsche.

In this article, we present a Nitsche-type method to impose a surface impedance function on an interface embedded with in a domain, where the wave propagation is governed by the Helmholtz equation for the acoustic pressure.
The method is conceptually similar to the approach of Hansbo \& Hansbo~\cite{Hansbo20043523} but accommodated to the special features of this wave propagation problem.
Our method is designed to seamlessly handle a complex-valued impedance function that is allowed to vanish, for which the method reduces to the symmetric interior-penalty method to enforce interelement continuity.
A condition that requires particular attention is when the surface impedance is \emph{stiffness dominated}.
The imaginary part of the surface impedance is then negative, which implies that \emph{surface waves} can occur in a layer close to the impedance layer.
The possibility of surface waves complicates the analysis of our method.
Nevertheless, we are able to show stability of the method, in terms of a discrete Gårding inequality, for a quite general class of surface impedance functions, under the condition, if applicable, that the surface waves are resolved by the mesh.

\section{Problem statement}\label{s:probstatement}

\subsection{Linear acoustics in the presence of impedance surfaces}\label{LinearAcoustics}

We consider time-harmonic acoustic wave propagation in still air.
The acoustic pressure and velocity are assumed to be given by $P(\bs x,t) = \Re e^{\i\omega t} p (\bs x)$ and $\bs U(\bs x,t) = \Re e^{\i\omega t} \bs u (\bs x)$, where $\omega\in\RR$ is the angular frequency, and where the acoustic pressure and velocity amplitude functions $p$ and $\bs u$ satisfies the linear, time-harmonic  wave equation, which in first-order form can be written
\begin{subequations}\label{e:firstordersyst}
\begin{align}
\i\omega \rho\bs u + \nabla p & = 0, \label{e:linmom} 	\\	
\frac{\i\omega}{c^2} p + \nabla\cdot \rho\bs u &= 0,
\end{align}
\end{subequations}
where $\rho$ is the static air density and $c$  the speed of sound.

We assume that there is a smooth, orientable surface $\Gamma$ located inside the domain, and we denote by $\bs n_1$ and $\bs n_2=-\bs n_1$ the two unit normal fields on each side of the surface.
We fix an orientation of the surface by selecting one of these normals and denoting it by $\bs n$. 
We assume that an acoustic flux $\bs n\cdot\bs u$ is transmitted (leaking) through the surface such that the acoustic flux at each point is proportional to the local acoustic pressure jump over the surface.
The pressure may thus be discontinuous across the surface although $\bs n\cdot\bs u$ is continuous.
Note that this model excludes transversal wave propagation in the surface material itself, since the model is strictly local. 
We define $p_i$, $i=1$, 2 as the limit acoustic pressure when approaching the surface from the interior of the side for which $\bs n_i$ is the outward-directed normal; that is, for $\bs x\in\Gamma$,
\begin{equation}\label{pidef}
p_i(\bs x) = \lim_{s\to0^+} p(\bs x-s\bs n_i(\bs x)) 
\end{equation}
and we denote the pressure jump over the surface by
\begin{equation}
\jump p = \bs n\cdot(\bs n_1p_1 + \bs n_2 p_2).
\end{equation}
Thus, under our modeling assumptions, the acoustic flux through the surface will satisfy
\begin{equation}\label{e:Zdef}
Z\bs n\cdot\bs u = \jump p.
\end{equation}
The frequency-dependent complex function $Z$ is the local \emph{transmission impedance} of the surface.
We assume that $\Re Z\geq0$, which means that the surface is \emph{acoustically passive}; that is, acoustic energy may be absorbed but not created by the surface.
The limits $|Z|\to 0$, $|Z|\to\infty$ model a vanishing and a sound-hard surface, respectively.
The condition $\Im Z =0$ means that the acoustic flux is in phase with the pressure jump, otherwise the surface will introduce a reactive load with a phase shift.
If the mechanical properties of the surface can be modeled by  distributed mass, spring, and damping densities, the transmission impedance will have the form
\begin{equation}\label{e:msdZ}
Z = d + \i(m\omega  - k/\omega),
\end{equation}
where $m$, $d$, and $k$ is the mass, damping, and spring constants per unit area, respectively.

The concept of transmission impedance is typically used as a macroscopic model for microscopic features.
For instance, a perforated metallic plate is often modeled as a mass--damping system, in which semi-empirical formulas for the mass and damping coefficients can be deduced from experiments~\cite{KiCu98}.
The reactive part of the perforation impedance can be established by homogenization of the inviscid equations~\cite{BiDrGm04}.

A special case is when $\Im Z < 0$.
As can be seen from expression~\eqref{e:msdZ}, this case corresponds to a surface whose acoustical properties are \emph{stiffness dominated}. 
In this case, \emph{surface waves}~\cite[\S~3.2.4]{RiHi15} can appear in a layer of depth $\delta = O(|\Im Z|/k)$ around the surface.
A local wave number associated with these waves increases with decreasing $|\Im Z|$, and approaches $O(1/\delta)$ as $\Im Z\to0$.

The acoustic velocity $\bs u$ can be reduced from system~\eqref{e:firstordersyst}, which leads to the Helmholtz equation for the acoustic pressure,
\begin{equation}
-\kappa^2 p  - \Delta p = 0,
\end{equation}
where $\kappa  = \omega/c$ is the (bulk) wave number.
Evaluating equation~\eqref{e:linmom} on either side of $\Gamma$, using the assumption that $\bs n\cdot\bs u$ is continuous over the surface,  we find that  the flux of the acoustic pressure is continuous over the surface,
\begin{equation}\label{contflux}
\frac{\partial p_1}{\partial n} =  \frac{\partial p_2}{\partial n}
\end{equation}
and  that 
\begin{equation}\label{e:linmomface}
\i \kappa \rho c\, \bs n\cdot\bs u  + \left\{\frac{\partial p}{\partial n} \right\} = 0,
\end{equation}
where
\begin{equation}
\left\{\frac{\partial p}{\partial n} \right\}  = \frac12\left(\frac{\partial p_1}{\partial n}+\frac{\partial p_2}{\partial n}\right).
\end{equation}
Making use of model~\eqref{e:Zdef}, expression~\eqref{e:linmomface} can be written as the transmission condition
\begin{equation}\label{impedancecond}
\frac{\i\kappa}{\zeta}\jump p + \left\{\frac{\partial p}{\partial n} \right\} = 0
\qquad\text{on $\Gamma$,}
\end{equation}
where $\zeta = Z/(\rho c)$ is the normalized transition impedance.

\subsection{Preliminaries}\label{sec:notation}

Let $\Omega$ be an open bounded domain in $\mathbb{R}^n, n = 2,3$ with a Lipschitz boundary $\partial\Omega.$
Assume that $\Omega$ can be split into two disjoint, open, and connected subdomains $\Omega_1$ and $\Omega_2 $ such that $\Omega = \Omega_1\cup\Omega_2\cup\Gamma_\textrm{I}$, where 
$\Gamma_\textrm{I} = \overline{\Omega}_1\cap\overline{\Omega}_2$ is a smooth interface boundary of codimension one with positive measure.
See Figure~\ref{fig:DomCases}  for illustrations.

Denote the space of square integrable functions on $\Omega_1\cup\Omega_2$ by $L^2(\Omega_1\cup\Omega_2)$.
Let $\alpha= (\alpha_1,\dots,\alpha_n)\in \mathbb{N}^n$ be a multi-index with $|\alpha|=\sum_{i=1}^n\alpha_i$.
For a nonnegative integer $m$,  $H^m(\Omega_1\cup\Omega_2)$ denotes the set of functions $p\in L^2(\Omega_1\cup\Omega_2)$ such that all weak partial derivatives $\partial^\alpha p$ with $|\alpha|\leq m$ are also in $L^2(\Omega_1\cup\Omega_2)$.
Spaces $L^2(\Omega_1\cup\Omega_2)$ and $H^m(\Omega_1\cup\Omega_2)$ are equipped  with norms
\begin{equation}
\begin{aligned}
\|p\|_{L^2(\Omega_1\cup\Omega_2)}^2 &=  \int\limits_{\Omega_1\cup\Omega_2}\; \mathclap{\lvert p\rvert^2,}
\\
  \|p\|_{H^m(\Omega_1\cup\Omega_2)}^2 &=  \sum_{\lvert\alpha\rvert\leq m}\Big( \int\limits_{\Omega_1}  \lvert\partial^\alpha p\rvert^2+\int\limits_{\Omega_2}  \lvert\partial^\alpha p\rvert^2\Big), 
\end{aligned}
\end{equation}
respectively.
Note that $\Omega_1\cup\Omega_2$ is disconnected as a topological space and that the $H^m(\Omega_1\cup\Omega_2)$ norms for $m\geq1$ are ``broken'' norms that exclude the interface and thus contain functions that are discontinuous over~$\Gamma_\text{I}$.

\begin{remark}
Throughout the article we do not explicitly specify the measure symbol (such as $dV$, for instance) in the integrals, since the type of measure will be clear from the domain of integration.
\end{remark}

For $i=1,2$, and a measurable subset $\Gamma \subset \partial \Omega_i$, we denote by
$\gamma_{\Gamma,m}^{\Omega_i}:H^{m+1}(\Omega_i)\to H^{m+1/2}(\Gamma )$ the continuous, $m$th-order trace operator~\cite[Theorem 8.7]{wloka1987}, for which there is a constant $c_l$ such that 
\begin{equation}\label{trace1}
\|\gamma_{\Gamma,m}^{\Omega_i}\, p\|_{H^{m+1/2}(\Gamma )}\leq c_1 \|p\|_{H^{m+1}(\Omega_i)}
\qquad\forall p\in H^{m+1}(\Omega_i).
\end{equation}
For $m=0$ and 1, operator $\gamma_{\Gamma,m}^{\Omega_i}$ applied on $C^{m}\bigl(\overline\Omega_i\bigr)$ functions yield the restrictions of $p$ and $\partial p/\partial n$ on $\Gamma$, respectively.
In addition to inequality~\eqref{trace1}, the zeroth-order trace operator satisfies~\cite[Theorem 1.6.6]{brenner2008}
\begin{equation}\label{trace2}
\|\gamma_{\Gamma,0}^{\Omega_i}\, p\|_{L^{2}(\Gamma )}^2\leq c_2 \|p\|_{L^2(\Omega_i)}\|p\|_{H^1(\Omega_i)},\;\;i=1,2,
\end{equation}
for some constant $c_2$. 
We denote the space of all measurable functions on $\Gamma_\textrm{I}$ that are bounded almost everywhere by
$L^{\infty}(\Gamma_\textrm{I})$.

\subsection{Model problem} \label{Prob_description}

Assume that the boundary of $\Omega_1\cup\Omega_2$ consists of the non-overlapping parts $\Gamma_{\textrm I}$, $\Gamma_{\textrm{io}}$, and $\Gamma_{\textrm{s}}$.
Using the notation $|\Gamma| = \int_{\Gamma}1\,\mathit{dS}$, we here limit the discussion to the following two cases for $\Gamma_{\textrm{io}}$,  illustrated in Figure~\ref{fig:DomCases}:
\begin{description}%\label{Domcases}
\item[\hspace{0.5cm}Case (i)]  $|\partial\Omega_i\cap\Gamma_{\textrm{io}}|>0, i = 1,2$ (Figure~\ref{fig:DomCases} (a)). That is, a portion of $\Gamma_{\textrm{io}}$ is part of the boundary to both subdomains.
\item [\hspace{0.5cm}Case (ii)] $| \partial\Omega_1\cap\Gamma_{\textrm{io}}|>0$ and $|\partial\Omega_2\cap\Gamma_{\textrm{io}}|=0$ (Figure~\ref{fig:DomCases} (b) or (c)). That is, $\Gamma_{\textrm{io}}$ is only a portion of the boundary of $\Omega_1$.
\end{description}
 \begin{figure}
 	\centering
 	\setlength{\unitlength}{0.5\textwidth}
 	\includegraphics[width=150mm]{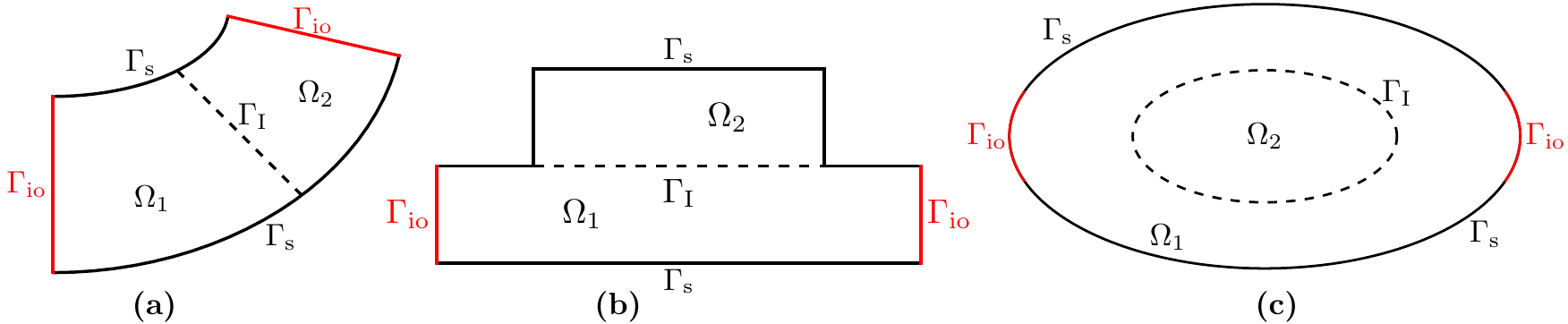} 
 	\caption{Computational domain cases}
 	\label{fig:DomCases} 
 \end{figure}

We consider the following boundary value problem for Helmholtz equation, in which impedance condition~\eqref{impedancecond} with $\zeta\neq 0$, $\text{Re }\zeta\geq0 $ is imposed on interface boundary $\Gamma_{\textrm{I}}$:
\begin{subequations}\label{stateequation}
\begin{alignat}{2}
\label{Helmholtzeqn}\Delta p+\kappa^2  p   &= 0                  &&\qquad	\textrm{ in }\Omega_1\cup\Omega_2, \\ 
\label{Gammaiocond} \textrm{i}\kappa p+\frac{\partial p}{\partial n}  &= 2\mathrm{i}\kappa g &&\qquad\textrm{ on } \Gamma_{\textrm{io}}, \\ 
\label{Gammascond} \frac{\partial p}{\partial n}  &= 0                                       &&\qquad\textrm{ on } \Gamma_{\textrm{s}}, \\ 
\label{GammaIcond}\frac{\i\kappa}{\zeta}\jump p + \left\{\frac{\partial p}{\partial n} \right\} &= 0&&\qquad\textrm{ on } \Gamma_{\textrm{I}}.
\end{alignat}
\end{subequations}

We will consider weak solutions to problem~\eqref{stateequation} in the space $H^1(\Omega_1\cup\Omega_2)$. 
Note that $\Omega_1\cup\Omega_2$ is disconnected  and that the elements in $H^1(\Omega_1\cup\Omega_2)$ are in general discontinuous across $\Gamma_{\textrm{I}}$.
A variational form of problem~\eqref{stateequation} may be formulated as follows: find $p \in H^1(\Omega_1\cup\Omega_2)$  such that
\begin{equation}\label{ContVarationalproblem}
\begin{aligned}
&\quad a(p,q) = \ell(q), \;\; \forall q\in H^1(\Omega_1\cup\Omega_2),
\end{aligned}
\end{equation}
where the linear functional $\ell$ is defined by
\begin{equation}\label{sourceterm}
\ell(q) = 2\mathrm{i}\kappa \int\limits_{\Gamma_{\mathrm{io}}} g {q}
\end{equation}
in which $g\in L^{2}(\Gamma_{\textrm{io}})$ is a given function, and the bilinear form $a$ is given by
\begin{align}\label{contbilinear_a} 
a(p,q) &= a_0(p,q)
   +\mathrm{i}\kappa \int\limits_{\Gamma_{\text{I}}} \frac{1}{\zeta} \jump p\llbracket {q}\rrbracket , 
\end{align} 
with
\begin{align}\label{bilinear0}
a_0(p,q) = \int\limits_{\mathclap{\Omega_1\cup\Omega_2}}  \nabla p\cdot\nabla  {q}  - \kappa^2  \int\limits_{\mathclap{\Omega_1\cup\Omega_2}}   p{q} 
+\textrm{i}\kappa  \int\limits_{\mathclap{\Gamma_{\textrm{io}} }}   p{q}.
\end{align}
\begin{remark}
For complex-valued problems, it is more common to consider \emph{sesquilinear} forms instead of, as here, bilinear forms.
We have chosen to define all problems in this article using bilinear forms since the expressions becomes slightly simpler (fewer complex conjugates will be needed).
\end{remark}
In order for bilinear form $a$ to be well defined on all of $H^1(\Omega_1\cup\Omega_2)\times H^1(\Omega_1\cup\Omega_2)$, we will in this section require that $\zeta\in L^{\infty}(\Gamma_{\textrm{I}})$ such that $|\zeta|\geq\zetalow$ almost everywhere for some constant $\zetalow>0$.
Note, however,  that $\zeta\equiv0$ corresponds to an interface that vanishes; $p$ and $\partial p/\partial n$ will be continuous across the interface in this case. 
It would be valuable to be able to treat this condition seamlessly in a numerical solution.  
The discrete variational problem introduced in \S~\ref{s:FEMethod} will be constructed in order to be less restrictive on the admissible impedance functions and will, in particular, accept vanishing $\zeta$ on the whole or parts of the interface.

Solutions to problem~\eqref{ContVarationalproblem} satisfy the balance law
\begin{equation}\label{energy4}
 \int_{\Gamma_{\textrm{io}} }  |g|^2 = \int_{\Gamma_{\textrm{io}} } |p-g|^2+ 
\int_{\Gamma_{\textrm{I}}}  \Re\big(\frac{1}{\zeta}\big)|\jump p|^2.
\end{equation}  
which can be derived from the imaginary part of equation~\eqref{ContVarationalproblem} with $q=\overline p$, where the overbar denotes complex conjugate.
Note that the second term on the right side of expression~\eqref{energy4} is nonnegative due to the assumption $\Re\zeta\geq 0$.
Balance law~\eqref{energy4} can be interpreted as saying that the power flowing into $\Omega_1\cup\Omega_2$ (the left side) equals the power flowing out plus the losses in the interface (the right side).

\subsection{Existence and Uniqueness}
Well-posedness of problem~\eqref{ContVarationalproblem} will be shown with the help of a Gårding inequality and compactness~\cite[\S~17.4]{wloka1987}.
The analysis is slightly nonstandard due to the presence of the interface integral in the bilinear form.
Let $\Gamma_\text{I}^-$ be the union of all subsets of $\Gamma_\text I$ in which $\Im \zeta<0$ almost everywhere.
(Recall that $\Im \zeta<0$ is the condition for the appearance of surface waves!)
Letting $V=H^1(\Omega_1\cup\Omega_2)$, we introduce the product space $U=L^2(\Omega_1\cup\Omega_2)\times L^2(\Gamma_{\text I}^-)$ and define the mapping $J: V\to U$ by  $q\mapsto \bigl(\sqrt{\kappa} q,\, \jump q\big|_{{\Gamma_{\text I}^-}}^{\null}/\sqrt{\delta_\zeta}\bigr)$, which means that
\begin{equation}\label{Unorm} 
\|J q\|^2_U = \kappa  \int\limits_{\mathclap{\Omega_1\cup\Omega_2 }}\lvert q\rvert^2+   \frac1\zetalow\int\limits_{\Gamma_{\text{I}}^-}  \big\lvert\jump q\big\rvert^2.
\end{equation} 
We note that $J$ is injective and compact with an image that is dense in $U$.
Density follows since smooth functions with compact support are dense in $L^2(\Omega_1\cup\Omega_2)$ and since for any $g\in L^2(\Gamma_{\text I}^-)$, there is a sequence of functions $q_n\in H^1(\Omega_1\cup\Omega_2)$ such that	$\jump {q_n}\big|_{{\Gamma_{\text I}^-}}^{\null}\to g$ in $L^2(\Gamma_{\text I}^-)$ and $q_n\to0$ in $L^2(\Omega_1\cup\Omega_2)$. 
Compactness follows from the fact that the embeddings $H^1(\Omega_1\cup\Omega_2)\to L^2(\Omega_1\cup\Omega_2)$ and $H^{1/2}(\Gamma_{\text I}^-)\to L^2(\Gamma_{\text I}^-)$ are compact together with the trace theorem on $\Gamma_{\text I}^-$.

With the help of mapping $J$, we may define the triple $V\stackrel{J}{\to} U \stackrel{J'}{\to} V'$, where $U$ is identified with its dual, and where the mapping $J$ and its dual $J'$ are injective, compact, and with images that are dense in $U$ and $V'$, respectively.   
To show a Gårding inequality with respect to this triple, take the real part of bilinear form~\eqref{contbilinear_a} with $q = \bar{p}$ to obtain
\begin{equation}\label{bilinear_real}
\begin{aligned}
\Re a(p,\bar{p}) &= \int\limits_{\mathclap{\Omega_1\cup\Omega_2}} \lvert\nabla p\rvert^2 - \kappa^2\int\limits_{\mathclap{\Omega_1\cup\Omega_2 } }   \lvert p\rvert^2
  + \kappa \int\limits_{\Gamma_{\textrm{I}}} \frac{ \Im\zeta}{\lvert\zeta\rvert^2}\lvert\jump p\rvert^2,
\\
&\geq 
\int\limits_{\mathclap{\Omega_1\cup\Omega_2}}  \lvert\nabla p\rvert^2 - \kappa^2\int\limits_{\mathclap{\Omega_1\cup\Omega_2 } }  \lvert p\rvert^2
  - \kappa \int\limits_{\Gamma_{\textrm{I}}^-} 
  \frac{\lvert\Im\zeta\rvert}{\lvert\zeta\rvert^2} \lvert\jump p \rvert^2,
\end{aligned}
\end{equation}
which satisfies
\begin{equation}\label{goardeq}
\begin{aligned}
\Re a(p,\bar{p}) +2\kappa\|J p\|^2_{U} & \geq \int\limits_{\mathclap{\Omega_1\cup\Omega_2}}
\lvert\nabla p\rvert^2 
+\kappa^2\int\limits_{\mathclap{\Omega_1\cup\Omega_2 } }\lvert p\rvert^2 
+ \kappa \int\limits_{\Gamma_{\textrm I}^-} \left(\frac2\zetalow -\frac{\lvert\Im\zeta\rvert}{\lvert\zeta\rvert^2}\right) \lvert\jump p\rvert^2\geq C\|p\|^2_{V},
\end{aligned}
\end{equation}
where $|\zeta|\geq\zetalow$ has been used in the last inequality and where $C=\min\{1,\kappa^2\}$.
Thus, the Fredholm alternative applies to bilinear form $a$, and there exists a unique solution of problem~\eqref{ContVarationalproblem} for each $\ell\in V'$ if uniqueness holds~\cite[Theorem~17.11]{wloka1987}. 
Uniqueness follows from the following lemma. 
\begin{lemma}\label{contuniquenesslemma}
If $\textrm{Re}(\zeta)\geq 0$ and $p\in H^1(\Omega_1\cup\Omega_2)$ such that
\begin{equation}\label{bilinearform_zero}
a(p,q) = 0\qquad \forall q \in H^1(\Omega_1\cup\Omega_2),
\end{equation}
then $p\equiv 0$ in both Case (i) and (ii) (as defined in \S~\ref{sec:notation}).
\end{lemma}
\begin{figure}\centering
\includegraphics{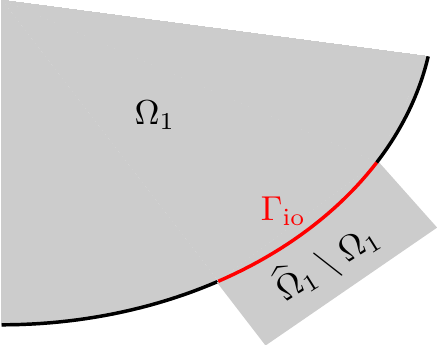}
\caption{In the proof of Lemma~\ref{contuniquenesslemma}, the solution $p$ in $\Omega_1$ is extended by zero into domain $\hat\Omega_1\supset\Omega_1$ which contains a portion $\hat\Omega_1\setminus\Omega_1$ laying outside of $\Gamma_\text{io}$}\label{f:extdomain}
\end{figure}

\begin{proof}
Let $a(p,q) =0$ $ \forall q \in H^1(\Omega_1\cup\Omega_2)$, then $a(p,\bar{p}) =0$. 
 By taking the imaginary part of $a(p,\bar{p})$, we find
\begin{equation} 
0=\textrm{Im}\;a(p,\bar{p}) =  \kappa \int\limits_{\Gamma_{\textrm{I}}} \frac{\textrm{Re}(\zeta)}{|\zeta|^2} |\jump p |^2 +\kappa  \int\limits_{\mathclap{\Gamma_{\textrm{io}} }}  |p|^2 .
\end{equation}
 Since $\textrm{Re}(\zeta)\geq 0$, we conclude that $p\equiv 0$ on $\Gamma_{\textrm{io}} $. 
Now choose $q\in H^1(\Omega_1\cup\Omega_2)$ such that $q|_{\Omega_2} =0$. Equation~\eqref{bilinearform_zero}  then reduces to
\begin{equation} \label{variationalprob_omega1}
  \int\limits_{\Omega_1}  \nabla p\cdot\nabla {q} - \kappa^2  \int\limits_{\Omega_1 }  p{q} 
   +\textrm{i}\kappa  \int\limits_{\Gamma_{\textrm{I}}} \frac{\overline{\zeta}}{|\zeta|^2} \jump p  {q}   =0 \;\; \forall q\in H^1(\Omega_1),
   \end{equation} 
with $p = 0$ on $\Gamma_\text{io}$. 
Let $\widehat\Omega_\text{1}$ be an extension of $\Omega_1$ into a strip outside $\Gamma_\text{io}\cap\partial\Omega_1$ (Figure~\ref{f:extdomain}), and extend $p$ by zero into $\widehat\Omega_\text{1}\setminus\Omega_1$.
Since $p$ vanishes over $\Gamma_\text{io}$, the extended function will be continuous over $\Gamma_\text{io}$. Thus, $p\in H^1(\widehat\Omega_\text{1})$  and, by equation~\eqref{variationalprob_omega1}, 
\begin{equation}
 \int\limits_{K}  \nabla p\cdot\nabla {q} - \kappa^2  \int\limits_{K}  p{q} 
   =0 \;\; \forall q\in C_0^\infty(K).
\end{equation} 
for each open set $K$ compactly embedded in $\widehat\Omega_1$ ($\overline K\subset \widehat\Omega_1 $), which implies that almost everywhere in each such $K$,
\begin{equation}\label{e:Helmholtz}
\Delta p + \kappa p = 0.
\end{equation}
Since $p$ satisfies equation~\eqref{e:Helmholtz} and vanishes identically  in $\widehat\Omega_1\setminus\Omega_1$, the unique continuation principle~\cite[Chap. 4.3]{leis2013} implies that $p\equiv 0$ in $\Omega_1$. 
In Case (i), that is, when $\lvert\partial\Omega_i\cap\Gamma_{\textrm{io}}\rvert>0,\; i = 1,2$, then, by the same argument as above, $p\equiv 0$ also in $\Omega_2$ and the conclusion follows.

In Case (ii), where $\lvert\partial\Omega_1\cap\Gamma_{\textrm{io}}\rvert>0$ and $\lvert\partial\Omega_2\cap\Gamma_{\textrm{io}}\rvert=0$, equation~\eqref{bilinearform_zero} becomes
\begin{equation}\label{variationalprob_omega2}
a(p,q) = \int\limits_{\mathclap{\Omega_2}} \nabla p\cdot\nabla {q}
- \kappa^2\int\limits_{\mathclap{\Omega_2} }   p{q}
   +\textrm{i}\kappa  \int\limits_{\Gamma_{\textrm{I}}} \frac{\overline{\zeta}}{|\zeta|^2} \llbracket {p}\rrbracket \llbracket {q}\rrbracket =0 \qquad \forall q\in H^1(\Omega_1\cup\Omega_2),
\end{equation}
since $p=0$ in $\Omega_1$.
By choosing  $q\in H^1(\Omega_1\cup\Omega_2)$ such that $q |_{\Omega_2} =0$, equation~\eqref{variationalprob_omega2} reduces to 
\begin{equation}
  \textrm{i}\kappa \int\limits_{\Gamma_{\textrm{I}}}\frac{\overline{\zeta}}{|\zeta|^2} \jump p  q  =0, \;\; \forall q\in H^1(\Omega_1),
  \end{equation}
which implies that $ \jump p=0$ on $\Gamma_{\text I}$, since $|\zeta|\geq \zetalow$.
Thus equation~\eqref{variationalprob_omega2} becomes
\begin{equation}
\int\limits_{\mathclap{\Omega_2}} \nabla p\cdot\nabla{q}- \kappa^2\int\limits_{\mathclap{\Omega_2} }   p{q} = 0 \;\; \forall q\in H^1(\Omega_1\cup\Omega_2),
\end{equation}
where $\gamma_{\Gamma_I}^{\Omega_2}p =0$ since $p\big|_{\Omega_1}^{\null}\equiv0$ and  $ \jump p=0$ on $\Gamma_{\text I}$.
Again, by an analogous extension argument as above together with the unique continuation principle, we conclude that $p\equiv 0$ in $\Omega_2$ and hence $p\equiv 0$ in $\Omega$ also in Case~(ii).
 \end{proof}
 
 \subsection{Continuity of fluxes}
 
Recall that \S~\ref{LinearAcoustics} started with the modeling assumption of an acoustic velocity that is continuous over the interface.
This assumption implied continuity of fluxes, expression~\eqref{contflux}, and led to formulation~\eqref{GammaIcond} of the interface condition in terms of the \emph{average} fluxes. 
However, it may not be obvious that the solution to corresponding variational problem~\eqref{ContVarationalproblem} in the end actually respects the a priori assumption of  continuous fluxes,  since continuity is not explicitly enforced.
In this section, we will see that the solution to the variational problem nevertheless satisfies such a continuity property in an appropriate weak sense.

The functional framework for the fluxes is the dual of the space $H^{1/2}_{00}(\Gamma_\text{I})$, a space that commonly occur in the context of transmission problems~\cite[Ch.~VII, \S~2.4, for instance]{DaLiV2}.
The space  $H^{1/2}_{00}(\Gamma_\text{I})$ is the space of traces on interface $\Gamma_\text{I}$ of functions in $H^1_0(\Omega)$ (recall from \S~\ref{sec:notation} that $\Omega=\Omega_1\cup\Omega_2\cup\Gamma_\text{I}$) provided with the norm
\begin{equation}\label{Hhalf00def}
\lVert\psi\rVert_{H^{1/2}_{00}(\Gamma_\text{I})} = \inf\left\{  \lVert w\rVert_{H^{1}(\Omega)}  \,\,\big\vert\, w\in H^1_0(\Omega), \gamma_{\Gamma_\text{I}}^\Omega w =\psi \right\}.
\end{equation}

To see how to define the weak flux, first assume that $p\in H^2(\Omega_1\cup\Omega_2)$ such that $\kappa^2p +\Delta p = 0$ in $\Omega_1$ and $p|_{\Omega_2}^{}\equiv0$. 
Integration by parts yields that for each $w\in H^1_0(\Omega)$, 
\begin{equation}\label{fluxmotivation}
\begin{aligned}
\int\limits_{\Gamma_\text{I}} \frac{\partial p_1}{\partial n_1} w 
&= \int\limits_{\mathclap{\Omega_1}}\Delta p \,w + \int\limits_{\Omega_1}\nabla p \cdot\nabla w 
= -\kappa^2\int\limits_{\Omega_1}p w + \int\limits_{\Omega_1}\nabla p \cdot\nabla w
\stackrel{\text{def}}{=\joinrel=}\ell_p(w), 
%\\
%\int\limits_{\Gamma_\text{I}} \frac{\partial p_2}{\partial n_2} w_2 
%&= -\kappa^2\int\limits_{\Omega_1}p w + \int\limits_{\Omega_1}\nabla p \cdot\nabla w.
\end{aligned}
\end{equation}
where $p_1$ is the limit of $p$ when approaching $\Gamma_\text{I}$ from the interior of $\Omega_1$, as in definition~\eqref{pidef}.
Now note that $\ell_p$ is a bounded linear functional on $H^1_0(\Omega)$; that is,  for each $p\in H^1(\Omega_1\cup\Omega_2)$, there is a constant $C_p$ such that $\lvert\ell_p(w)\rvert\leq C_p\lVert w\rVert_{H^1(\Omega)}$ for each $w\in H^1_0(\Omega)$.
In particular, $\ell_p$ is thus a linear functional on $H^{1/2}_{00}(\Gamma_\text{I})$ by definition~\eqref{Hhalf00def}.
Thus, if $p\in H^1(\Omega_1\cup\Omega_2)$ is  the solution of variational problem~\eqref{ContVarationalproblem}, 
there is a functional $\lambda_1\in H^{1/2}_{00}(\Gamma_\text{I})'$ such that, for each $w\in H^1_0(\Omega)$,
\begin{equation}\label{lambda1def}
\langle\lambda_1, w\rangle_{\Gamma_\text{I}} = -\kappa^2\int\limits_{\Omega_1}p w + \int\limits_{\Omega_1}\nabla p \cdot\nabla w,
\end{equation}
where $\langle\cdot, \cdot\rangle_{\Gamma_\text{I}}$ denotes the duality pairing on $H^{1/2}_{00}(\Gamma_\text{I})$.
By expression~\eqref{fluxmotivation}, we see that the weak flux $\lambda_1$ is a generalization of $\partial p_1/\partial n_1$ for weak solutions $p$.

Analogously, assuming that $p\in H^2(\Omega_1\cup\Omega_2)$ such that $\kappa^2p +\Delta p = 0$ in $\Omega_2$ and $p|_{\Omega_1}^{}\equiv0$, integration by parts yields that, for each $w\in H^1_0(\Omega)$,
\begin{equation}\label{fluxmotivation2}
\begin{aligned}
\int\limits_{\Gamma_\text{I}} \frac{\partial p_2}{\partial n_2} w_2 
=- \int\limits_{\Gamma_\text{I}} \frac{\partial p_2}{\partial n_1} w_2 
= -\kappa^2\int\limits_{\Omega_2}p w + \int\limits_{\Omega_2}\nabla p \cdot\nabla w.
\end{aligned}
\end{equation}
Hence, for $p$ being the solution of variational problem~\eqref{ContVarationalproblem}, there is thus a $\lambda_2\in H^{1/2}_{00}(\Gamma_\text{I})'$ such that, for each $w\in H^1_0(\Omega)$,
\begin{equation}\label{lambda2def}
-\langle\lambda_2, w\rangle_{\Gamma_\text{I}} = -\kappa^2\int\limits_{\Omega_2}p w + \int\limits_{\Omega_2}\nabla p \cdot\nabla w,
\end{equation}
and  $\lambda_2$ is thus a generalization of $\partial p_2/\partial n_1$.

Addition of expressions~\eqref{lambda1def} and \eqref{lambda2def} implies that for each $w\in H^1_0(\Omega)$,
\begin{equation}
\langle\lambda_1-\lambda_2, w\rangle_{\Gamma_\text{I}} = -\kappa^2\int\limits_{\mathclap{\Omega_1\cup\Omega_2}}p w + \int\limits_{\mathclap{\Omega_1\cup\Omega_2}   }\nabla p \cdot\nabla w
= - \int\limits_{\Gamma_\text{I}} \frac1{\zeta}\jump p \jump w = 0,
\end{equation}
where equation~\eqref{ContVarationalproblem} has been used in the second equality, and the fact that $\jump w = 0$ for each $w\in H^1(\Omega)$ in the third equality.
Thus, $\lambda_1=\lambda_2$, that is, the flux  is continuous over the interface.

\section{Discrete variational problem}\label{s:FEMethod}

As long as the condition $|\zeta|\geq\zetalow$ is respected, a finite element discretization can directly be applied to variational form~\eqref{ContVarationalproblem}.
However, this discretization will not be able to handle an interface that vanishes completely, and the condition number of the matrix will blow up if $\zeta\to0$ on a set of positive measure on $\Gamma_\text I$.  
To handle the case of a vanishing and non-vanishing interface in a common formulation, we introduce a new variational form of Nitsche type.
The method is based on an idea previously proposed to treat compliant interfaces in solid mechanics~\cite{Hansbo434217}.

In this section, we assume that $\Omega\subset\RR^3$ and, in order to avoid domain approximations, that $\Omega$ has a polyhedral boundary. 
We introduce families of separate, non-degenerate and quasi-uniform tetrahedral triangulations $\bigl\{\mathcal{T}_1^h\bigr\}_{h>0}$ and $\bigl\{\mathcal{T}_2^h\bigr\}_{h>0}$ of $\Omega_1$ and $\Omega_2$, respectively, parameterized by $h =\max_{K\in \mathcal{T}_1^h\cup\mathcal{T}_2^h} h_K$, where $h_K$ is the diameter of  element $K\in \mathcal{T}_1^h\cup\mathcal{T}_2^h$. 
The mesh nodes of the two triangulations do not need to match on the interface $\Gamma_{\textrm{I}}$.
We define the finite element space $V_h= V_1^h + V_2^h$,  where $V_i^h$ consists functions that are continuous on $\Omega_i$, polynomials on each element in $\mathcal{T}_i^h$, and extended by zero into $\Omega\setminus\Omega_i$; that is, 
\begin{equation}
V_i^h = \{v\in H^1(\Omega_1\cup\Omega_2) \mid v\big|_{K}^{\strut}\in P_k(K),\, \forall K\in \mathcal{T}_i^h \textrm{ and } v\equiv 0\; \mathrm{ otherwise} \},
\end{equation}
where $P_k(K)$ denotes the polynomials of maximum degree $k\geq 1$ on element $K$.

For each element $K\in \mathcal{T}_1^h\cup\mathcal{T}_2^h$, we let $\rho_K$ denote the diameter of the largest ball contained in $K$.
The condition of non degeneracy is that there exists a constant $C$ such that for each $h>0$ and $K\in\mathcal{T}_1^h\cup\mathcal{T}_1^h$, $h_K/\rho_K\le C$.

\begin{remark}\label{r_nondegen}
We note that if an element $K$ satisfies condition $h_K/\rho_K\le C$, then all faces $F$ of $K$ will satisfy the condition $h_F/\rho_F\le C$, where $h_F$ is the diameter of $F$ and $\rho_F$ is the diameter of the largest disc contained in $F$.
That is, if the volume mesh family is non degenerate, then the surface mesh family on $\Gamma_{\text I}$, generated by the faces of the mesh that intersect with $\Gamma_{\text I}$, is also non degenerate. 
The above statement is a consequence of that $h_K\ge h_F$ and $\rho_K\le\rho_F$ for any face $F$ of an element $K$.
The first inequality follows since $F$ is a face of $K$ and thus is included in any ball that contains $K$.
To verify that $\rho_K \le \rho_F$, as illustrated in Figure~\ref{fig:degen}, consider the plane $P_F$ that is parallel to $F$ and that passes through the center of the largest ball in $K$.
The intersection $P_F\cap K$ is a triangle that by construction can hold a disc of diameter $\rho_K$.
This triangle can be translated so that is becomes a subset of $F$, and thus $F$ can also hold a disc of diameter $\rho_K$, which entails that $\rho_F\ge\rho_K$.
\end{remark}

\begin{figure}
 \centering  
  \includegraphics[scale=1]{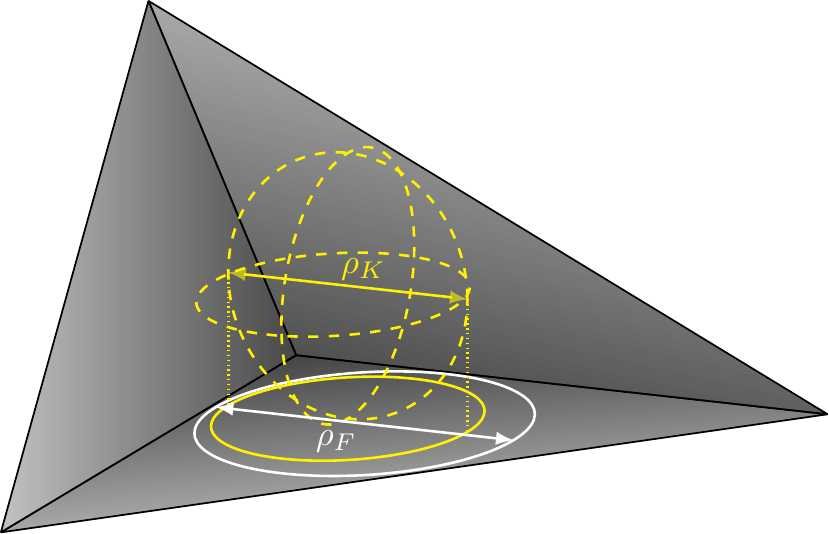}
  \caption{The diameter $\rho_K$ of the largest ball inscribed in a tetrahedron is smaller or equal to the diameter $\rho_F$ of the largest disc inscribed in any face $F$ of the tetrahedron.} 
  \label{fig:degen}
\end{figure}

To motivate the proposed variational form, assume that $p\in H^2(\Omega_1\cup\Omega_2)$ satisfies boundary value problem~\eqref{stateequation}. 
Multiplying Helmholtz equation~\eqref{Helmholtzeqn} by a test function $q\in H^1(\Omega_1\cup\Omega_2)$ and applying integration by parts, using boundary conditions~\eqref{Gammaiocond} and~\eqref{Gammascond} together with the continuity of the acoustic flux on the interface boundary, we find that
\begin{equation}\label{discbilinearconst1}
a_0(p,q)  -  \int\limits_{\Gamma_{\textrm{I}}} \left\{\frac{\partial p}{\partial n}\right\}\llbracket { q}\rrbracket = \ell(q),
\end{equation}
 where $\ell$ and $a_o$ are as stated in definitions~\eqref{sourceterm} and~\eqref{bilinear0}. 
After addition and subtraction of $\frac{\zeta}{ \textrm{i}\kappa}\!\!\left\{\frac{\partial p}{\partial n}\right\}\!\!\left\{\frac{\partial q}{\partial n}\right\}$, expression~\eqref{discbilinearconst1} can be written as  
\begin{align}\label{discbilinearconst2}
 a_0(p,q)- \int\limits_{\Gamma_{\textrm{I}}} \left\{\frac{\partial p}{\partial n}\right\}\left(\llbracket { q}\rrbracket+\frac{\zeta}{ \textrm{i}\kappa}\left\{\frac{{\partial q}}{\partial n}\right\}\right) +\int\limits_{\Gamma_{\textrm{I}}}  {\frac{\zeta}{ \textrm{i}\kappa}} \left\{\frac{\partial p}{\partial n}\right\}\left\{\frac{\partial q}{\partial n}\right\}= \ell(q).
\end{align}
Since $p$ satisfies boundary condition~\eqref{GammaIcond}, equation~\eqref{discbilinearconst2} can be extended to 
\begin{align}\label{discretebilinear}
a_{\lambda}(p,q): = a_0&(p,q) -   \int\limits_{\Gamma_{\textrm{I}}} \left\{\frac{\partial p}{\partial n}\right\} \left(\llbracket { q}\rrbracket+\frac{\zeta}{ \textrm{i}\kappa}\left\{\frac{{\partial q}}{\partial n}\right\}\right) - \int\limits_{\Gamma_{\textrm{I}}} \left(\jump p+\frac{\zeta}{ \textrm{i}\kappa}\left\{\frac{\partial p}{\partial n}\right\}\right)\left\{\frac{\partial q}{\partial n}\right\}
 \nonumber \\
 +& \int\limits_{\Gamma_{\textrm{I}}}  {\frac{\zeta}{ \textrm{i}\kappa}} \left\{\frac{\partial p}{\partial n}\right\}\left\{\frac{\partial q}{\partial n}\right\}
 + \int\limits_{\Gamma_{\textrm{I}}}\lambda\left(\jump p+\frac{\zeta}{ \textrm{i}\kappa}\left\{\frac{\partial p}{\partial n}\right\}\right) \left(\llbracket { q}\rrbracket+\frac{\zeta}{ \textrm{i}\kappa}\left\{\frac{{\partial q}}{\partial n}\right\}\right)   = \ell(q),
\end{align}
for any complex-valued $\lambda\in L^{\infty}(\Gamma_{\textrm{I}})$.

The method we propose, based on variational expression~\eqref{discretebilinear}, is: %~\cite{Nitsche}
find $p_h \in V_h$   such that
\begin{equation}\label{discretestateeqn}
\begin{aligned}
\quad a_{\lambda}(p_h ,q_h )  &= \ell(q_h), \;\; \forall q_h \in V_h.
\end{aligned}
\end{equation}
The construction of the method immediately implies the following consistency lemma.
\begin{lemma}\label{consistentlemma}
A solution $p\in H^2(\Omega_1\cup\Omega_2)$  of problem~\eqref{ContVarationalproblem} satisfies
\begin{equation}
a_{\lambda}(p,q_h) = \ell(q_h) \qquad\forall q_h\in V_h.
\end{equation}
\end{lemma}

Lemma~\ref{consistentlemma} implies Galerkin orthogonality in the sense that if $p\in H^2(\Omega_1\cup\Omega_2)$ solves problem~\eqref{ContVarationalproblem} and $p_h$ solves equation~\eqref{discretestateeqn}, then
\begin{equation}\label{orthogocond}
a_{\lambda}(p-p_h,q_h) = 0 \qquad\forall q_h\in V_h.
\end{equation}

We will choose  $\lambda$ as a complex-valued function of the local acoustical impedance $\zeta$, the wave number  $\kappa$, the mesh size $h$, and a sufficiently large parameter $\gamma>0$ to be specified later,
\begin{equation}\label{lambdadfn}
\lambda  = \left(\frac{h}{\gamma} + \frac{\zeta}{\textrm{i}\kappa}\right)^{-1},
\end{equation}
which under the requirements specified below  will be a nonzero and bounded function.

Under definition~\eqref{lambdadfn}, we note that, formally, for $h=0$ and $\zeta\neq0$, we obtain that $a_\lambda = a$; that is, the variational problem using $a_\lambda$ then reduces to the standard problem~\eqref{ContVarationalproblem}.
In the other extreme case, for $\zeta\equiv0$, $h>0$,
\begin{equation}
\begin{aligned}
a_\lambda(q_h,p_h) &= \int\limits_{\mathclap{\Omega_1\cup\Omega_2}}\nabla q_h\cdot\nabla p_h - \kappa^2\int\limits_{\mathclap{\Omega_1\cup\Omega_2}}q_hp_h 
+ \i\kappa\int\limits_{\Gamma_\text{io}}q_hp_h
\\ &\qquad
-\int\limits_{\Gamma_{\text I}}\jump{q_h}\Bigl\{\frac{\partial p_h}{\partial n}\Bigr\}
-\int\limits_{\Gamma_{\text I}}\Bigl\{\frac{\partial p_h}{\partial n}\Bigr\}\jump{p_h}
%\\ &\qquad
+\int\limits_{\Gamma_{\text I}}\frac\gamma{h}\jump{q_h}\jump{p_h};
\end{aligned}
\end{equation}
that is, the variational problem with $a_\lambda$ then reduces to the standard Nitsche method to weakly impose continuity of $p_h$ and $\partial p_h/\partial n$ over $\Gamma_{\text I}$. 
Variational problem~\eqref{discretestateeqn} with $\lambda$ defined as in expression~\eqref{lambdadfn} can thus be interpreted as an interpolation between these two extreme cases.

Throughout the following, we will require the condition
\begin{equation}\label{e:zetahcond}
h\Im \zeta\geq -\frac{\gamma}{4\kappa}|\zeta |^2\qquad \text{almost everywhere on $\Gamma_\text I$.}
\end{equation}
The next result yields sufficient conditions to satisfy requirement~\eqref{e:zetahcond}. 
\begin{lemma}\label{l:suffcond} 
Let $\zeta\in L^\infty(\Gamma_I)$ such that $|\zeta| \geq\zetalow >0 $ almost everywhere on $\Gamma^-_\text I$.
Then there is an $h_0>0$ such that condition~\eqref{e:zetahcond} is satisfied for each $0<h\leq h_0$.
In particular,
\begin{itemize}
\item[(i)]
when  $|\Gamma^-_\text I| = 0$ condition~\eqref{e:zetahcond} will be satisfied  $\forall h>0$ ($h_0=+\infty$), and
\item[(ii)]
when $|\Gamma^-_\text I| > 0$, condition~\eqref{e:zetahcond} will be satisfied if 
\begin{equation}\label{e:hresolved}
0<h\leq h_0 = \frac{\gamma\zetalow}{4\kappa}.
\end{equation}
\end{itemize}
\end{lemma}
\begin{proof}
When  $|\Gamma^-_\text I| = 0$, $\Im\zeta\geq0$ almost everywhere on $\Gamma_\text I$ and inequality~\eqref{e:zetahcond} will therefore be satisfied $\forall h>0$.
In regions where $\Im\zeta<0$, condition~\eqref{e:zetahcond} is equivalent to
\begin{equation}\label{e:zetahcond2}
h\leq \frac{\gamma|\zeta|^2}{4\kappa|\Im\zeta|}.
\end{equation}
Since $|\zeta| \geq\zetalow >0 $ almost everywhere on $\Gamma^-_\text I$, 
\begin{equation}
\zetalow\leq\frac{|\zeta|^2}{|\Im\zeta|}
\end{equation}
on $\Gamma^-_\text I$, so if $h$ satisfies condition~\eqref{e:hresolved},  we conclude that
\begin{equation}
0<h\leq \frac{\gamma\zetalow}{4\kappa}\leq \frac{\gamma|\zeta|^2}{4\kappa|\Im\zeta|};
\end{equation}
almost everywhere on $\Gamma^-_\text I $; that is, condition~\eqref{e:zetahcond2}  will be satisfied.
\end{proof}

Recall that $\Gamma^-_\text I$  is characterized by $\Im\zeta<0$, that surface waves can appear in this case, and that the thickness of the surface wave layer is $O(|\Im\zeta|/\kappa)$. 
Condition~\eqref{e:hresolved} can thus be interpreted as simply saying that the surface wave layer has to be resolved by the mesh. 

The following lemma shows some properties of $\lambda$ that will be used to show continuity and coercivity of $a_\lambda$ in Theorems~\ref{discContinequality} and~\ref{discGordinginequality} below.
\begin{lemma}\label{lambdacondn}
Let $\gamma>0$ be given and let $\zeta\in L^\infty(\Gamma_I)$ such that $|\zeta| \geq\zetalow >0 $ almost everywhere on $\Gamma^-_\text I$.
Then there is an $h_0>0$ such that, for each $0<h\leq h_0$, the function $\lambda$ in definition~\eqref{lambdadfn} satisfies
\begin{subequations}
\begin{alignat}{2}
\label{lambdainequal1} 0< |\lambda | &\leq \frac{\gamma}{h} &&\text{a.e.\ on $\Gamma_\text I$,}\\
\label{lambdaidentity1} 1-\lambda \frac{\zeta}{ \textrm{i}\kappa} &= \frac{h}{\gamma} \lambda 
					&&\text{a.e.\ on $\Gamma_\text I$,}\\
\label{lambdainequal2}\Big|\frac{\zeta}{ \textrm{i}\kappa}\Big(1-\lambda  \frac{\zeta}{ \textrm{i}\kappa}  \Big)\Big| 
& \leq  \frac{2h}{\gamma}&&\text{a.e.\ on $\Gamma_\text I$,} \\
\label{lambdaminbound} \lvert\lambda\rvert &\leq \frac{2\kappa}{\zetalow}
&&\text{a.e.\ on $\Gamma^-_\text I$,}\\ 
\label{lambdasplitbound}
\Re\lambda + \Im\lambda -\frac12\lvert\lambda\rvert
&\geq
\begin{cases}
\frac12\lvert\lambda\rvert &\text{a.e.\ on $\Gamma_\text I\setminus\Gamma^-_\text I$,}
\\[5pt]
\displaystyle-3\frac{\kappa}{\zetalow}        &\text{a.e. on $\Gamma^-_\text I$.}
\end{cases}
\end{alignat} 
\end{subequations}

\end{lemma}
\begin{proof}
The given assumptions imply that Lemma~\ref{l:suffcond} applies, and there is thus a $h_0>0$ such that condition~\eqref{e:zetahcond} holds for $0<h\leq h_0$. 
The denominator of $|\lambda|^2$, using the proposed definition~\eqref{lambdadfn}, satisfies
\begin{equation}\label{e:ldenomineq}
\left|\frac{h}\gamma + \frac{\zeta}{\i\kappa}\right|^2 = 
\left(\frac{h}\gamma + \frac{\Im\zeta}{\kappa}\right)^2 + \left(\frac{\Re\zeta}{\kappa}\right)^2
= \frac{h^2}{\gamma^2} + \frac{\lvert\zeta\rvert^2}{\kappa^2} + 2\frac{h\Im\zeta}{\gamma\kappa}
\geq \frac{h^2}{\gamma^2} + \frac{\lvert\zeta\rvert^2}{2\kappa^2}>0,
\end{equation}
where condition~\eqref{e:zetahcond} and the fact that $h>0$ have been used in the first and second inequality, respectively. 
Expression~\eqref{e:ldenomineq} shows that the $\lambda$ in definition~\eqref{lambdadfn} is well defined and satisfies the right inequality in~\eqref{lambdainequal1}.
Identity~\eqref{lambdaidentity1} then follows immediately from the definition of $\lambda$. 
The left inequality in~\eqref{lambdainequal1} follows from  that
\begin{equation}
\lvert\lambda\rvert^{-1} = \left|\frac{h}\gamma + \frac{\zeta}{\i\kappa}\right|\leq \frac{h}\gamma + \frac{|\zeta|}{\kappa}
\end{equation}
is bounded.

Moreover, by identity~\eqref{lambdaidentity1}, we have
\begin{equation} 
\frac{\zeta}{ \textrm{i}\kappa}\Big(1-\lambda  \frac{\zeta}{ \textrm{i}\kappa}  \Big)  = \frac{\zeta}{ \textrm{i}\kappa}  \frac{h}{ \gamma} \lambda = \frac{h}{ \gamma} \Big(\frac{h}{ \gamma}+\frac{\zeta}{ \textrm{i}\kappa}  -\frac{h}{ \gamma} \Big)\lambda = \frac{h}{ \gamma} \Big(\lambda^{-1} -\frac{h}{ \gamma} \Big)\lambda = \frac{h}{ \gamma}  - \Big(\frac{h}{\gamma}\Big)^2 \lambda,
\end{equation} 
which, by the triangle inequality and inequality~\eqref{lambdainequal1}, yields the bound~\eqref{lambdainequal2}, that is,
\begin{equation}
\Big|\frac{\zeta}{ \textrm{i}\kappa}\Big(1-\lambda  \frac{\zeta}{ \textrm{i}\kappa}  \Big)\Big| \leq \frac{h}{ \gamma}  +\Big(\frac{h}{\gamma}\Big)^2 |\lambda| \leq  \frac{2h}{\gamma}.
\end{equation} 

Inequality~\eqref{e:ldenomineq} implies the bound
\begin{equation}
\left|\frac{h}\gamma + \frac{\zeta}{\i\kappa}\right| \geq \frac{\lvert\zeta\rvert}{\sqrt{2}\kappa}.
\end{equation}
Thus, on $\Gamma^-_\text I$, 
\begin{equation}\label{lambdaGMbound}
\lvert\lambda\rvert \leq\frac{\sqrt{2}\kappa}{\lvert\zeta\rvert} \leq \frac{\sqrt{2}\kappa}{\zetalow}%< +\infty,
\end{equation} 
where the second inequality follows from that $\lvert\zeta\rvert\geq\zetalow$.
Thus inequality~\eqref{lambdaminbound} holds.

To show the last bounds, we first note that, from the definition of $\lambda$,
\begin{equation}\label{lambdaRI} 
 \lambda =  \frac{{h}/{\gamma}+\overline{{\zeta}/{ \textrm{i}\kappa}}}{|{h}/{\gamma} + {\zeta}/{\textrm{i}\kappa}|^2} 
 = \Bigl(\frac{h}{\gamma}+ \frac{\Im\zeta}{ \kappa}+\textrm{i}\frac{\Re\zeta}{ \kappa} \Bigr){|\lambda|^2},
\end{equation} 
which means that a.e.\ on $\Gamma_\text I\setminus\Gamma^-_\text I$, since $\Im\zeta\geq0$ and $\Re\zeta\geq0$ and $\lvert\Im z\rvert + \lvert\Re z\rvert\geq\lvert z\rvert$,
\begin{equation}
\Re\lambda + \Im\lambda -\frac12\lvert\lambda\rvert =  \lvert\Re\lambda\rvert + \lvert\Im\lambda\rvert -\frac12\lvert\lambda\rvert\geq \frac12\lvert\lambda\rvert.
%\qquad\text{}
\end{equation}

On $\Gamma^-_\text I$, expression~\eqref{lambdaRI} together with the bound~\eqref{lambdaGMbound}, the fact that $\Im\zeta <0$, $\Re\zeta\geq0$, and $\lvert\zeta\rvert\geq\zetalow$  in that region, imply that
\begin{align}
\Re\lambda + \Im\lambda -\frac12\lvert\lambda\rvert
&=  \Bigl(\frac{h}{\gamma}+ \frac{\Im\zeta}{ \kappa}\Bigr)\lvert\lambda\rvert^2
+ \frac{\Re\zeta}{ \kappa}{\lvert\lambda\rvert^2} -\frac12\lvert\lambda\rvert
\nonumber\\
&\geq -\frac{\lvert\Im\zeta\rvert}{\kappa}\lvert\lambda\rvert^2 -\frac12\lvert\lambda\rvert
\geq -3\frac{\kappa}{\zetalow}.
\end{align}
 \end{proof}

The analysis of the method will be carried out in the mesh and wave number dependent norm 
\begin{align}\label{triplenorm}
|||\, p_h\, |||^2 = \int\limits_{\mathclap{\Omega_1\cup\Omega_2}} \lvert\nabla p_h\rvert^2  + \kappa^2\int\limits_{\mathclap{\Omega_1\cup\Omega_2}} \lvert p_h\rvert^2  +  \frac{1}{\gamma}  \int\limits_{\Gamma_{\textrm{I}}}   h\Big\lvert\left\{\frac{\partial p_h}{\partial n}\right\}\Big\rvert^2+ \int\limits_{\Gamma_{\textrm{I}}}\lvert\lambda\rvert\big\lvert\llbracket { p_h}\rrbracket\big\rvert^2.
\end{align}   
Note that the coefficient in the last integral does not vanish due to inequality~\eqref{lambdainequal1}.

We will make use the following standard inverse inequality, whose proof relies on the mesh being quasi uniform.
\begin{lemma}\label{l:inverseineq}
For $p_h\in V_h$ there exist a constant $C_I>0$ such that
\begin{equation}\label{inverseineq}
\int\limits_{\Gamma_{\textrm{I}}}  h  \Big\lvert\left\{\frac{\partial p_h}{\partial n}\right\}\Big\rvert^2\leq C_I \int\limits_{\mathclap{\Omega_1\cup\Omega_2}}\lvert\nabla p_h\rvert^2,
\end{equation}
where $C_I$ depends on the polynomial approximation  order and the mesh regularity and quasi-uniformity constants.
\end{lemma}
Warburton and Hesthaven~\cite{WaHe03} provide proofs for inverse estimates of this type for triangular and tetrahedral meshes and present explicit expressions for the constant $C_I$.

For the purpose of analysis, we rewrite bilinear form~\eqref{discretebilinear} in the following way,
\begin{align}\label{discretebilinear2}
a_{\lambda}(p_h,q_h)& = a_{0}(p_h,q_h)
 - \int\limits_{\Gamma_{\textrm{I}}} (1-\lambda\frac{\zeta}{ \textrm{i}\kappa})\llbracket { q_h}\rrbracket\left\{\frac{\partial p_h}{\partial n}\right\} 
-  \int\limits_{\Gamma_{\textrm{I}}} \big(1-\lambda\frac{\zeta}{ \textrm{i}\kappa} \big)\llbracket p_h\rrbracket\left\{\frac{\partial q_h}{\partial n}\right\}  \nonumber\\
&\quad    -\int\limits_{\Gamma_{\textrm{I}}}    \frac{\zeta}{ \textrm{i}\kappa}\big(1-\lambda  \frac{\zeta}{ \textrm{i}\kappa}  \big)\left\{\frac{\partial p_h}{\partial n}\right\}\left\{\frac{\partial q_h}{\partial n}\right\} +   \int\limits_{\Gamma_{\textrm{I}}} \lambda\llbracket p_h\rrbracket\llbracket {q_h}\rrbracket .
\end{align}

\begin{theorem}[Continuity]\label{discContinequality}
Let $\lambda$ be as in definition~\eqref{lambdadfn} and assume that condition~\eqref{e:zetahcond} is satisfied. 
Then there is a  constant $C_c$,  dependent on $\kappa$ and trace inequality constant  such that
\begin{equation}
\lvert a_{\lambda}(p_h,q_h)\rvert \leq  C_c\,|||p_h|||\;|||q_h|||\qquad\forall p_h,q_h\in V_h\cup H^2(\Omega_1\cup\Omega_2).
\end{equation}
\end{theorem}

 \begin{proof}
Using identity~\eqref{lambdaidentity1}, Cauchy--Schwarz inequality, and inequality~\eqref{lambdainequal1}, we bound the second term on the right side of expression~\eqref{discretebilinear2} as follows,
\begin{align}\label{discretebilinear22}
\biggl| \int\limits_{\Gamma_{\textrm{I}}} \bigl(1-\lambda\frac{\zeta}{ \textrm{i}\kappa}\bigr)\llbracket { q_h}\rrbracket\left\{\frac{\partial p_h}{\partial n}\right\} \biggr| 
&\leq  \int\limits_{\Gamma_{\textrm{I}}} \frac{h}{\gamma}\lvert\lambda\rvert\bigg|\llbracket { q_h}\rrbracket\left\{\frac{\partial p_h}{\partial n}\right\} \biggr| 
\leq  \Bigl(\int\limits_{\Gamma_{\textrm{I}}}  \frac{h}{\gamma}|\lambda |^2\big|\llbracket { q_h}\rrbracket\big|^2 \int\limits_{\Gamma_{\textrm{I}}} \frac{h}{\gamma}\Big|\left\{\frac{\partial p_h}{\partial n}\right\} \Big|^2 \Bigr)^{1/2} \nonumber\\
&\leq  \Bigl(  \int\limits_{\Gamma_{\textrm{I}}} \lvert\lambda\rvert\bigl|\llbracket { q_h}\rrbracket\bigr|^2 \int\limits_{\Gamma_{\textrm{I}}} \frac{h}{\gamma}\biggl|\left\{\frac{\partial p_h}{\partial n}\right\}\biggr|^2 \Bigr)^{1/2}.
\end{align} 
A bound for the third term can be obtained similarly.  
For the fourth term, using inequality~\eqref{lambdainequal2} and the Cauchy--Schwarz inequality, we get
\begin{equation}%\label{discretebilinear23}
\biggl| \int\limits_{\Gamma_{\textrm{I}}}  \frac{\zeta}{ \textrm{i}\kappa}\bigl(1-\lambda  \frac{\zeta}{ \i\kappa}\bigr) \left\{\frac{\partial p_h}{\partial n}\right\}\left\{\frac{\partial q_h}{\partial n}\right\}\biggr| 
\leq
2\Bigl(\int\limits_{\Gamma_{\textrm{I}}}  \frac{h}{\gamma} \biggl|\left\{\frac{\partial p_h}{\partial n}\right\}\biggr|^2    
\int\limits_{\Gamma_{\textrm{I}}}\frac{h}{\gamma} \biggl|\left\{\frac{\partial q_h}{\partial n}\right\}\biggr|^2 \Bigr)^{1/2}.
 \end{equation} 
By  Cauchy--Schwarz inequality, we find the following bound of the fifth term,
  \begin{align}%\label{discretebilinear24}
\biggl|\int\limits_{\Gamma_{\textrm{I}}} \lambda \llbracket p_h\rrbracket\llbracket {q_h}\rrbracket\biggr|  
\leq \Bigl( \int\limits_{\Gamma_{\textrm{I}}}\lvert\lambda|\bigl\rvert\llbracket { p_h}\rrbracket\bigr|^2  \int\limits_{\Gamma_{\textrm{I}}}\lvert\lambda\rvert\bigl|\llbracket { q_h}\rrbracket\bigr|^2 \Bigr)^{1/2}.
 \end{align}
Finally, using trace inequality~\eqref{trace1} and Cauchy--Schwarz inequality, we obtain that
\begin{align}\label{discretebilinear25}
|a_0(p_h,q_h)| & \leq c_{\kappa}\|p_h\|_{H^1{(\Omega_1\cup\Omega_2})}\|q_h\|_{H^1{(\Omega_1\cup\Omega_2)}},
 \end{align}  
where $c_{\kappa}$ depends on $\kappa$ and the trace inequality constant.
The conclusion then follows from bounds~\eqref{discretebilinear22}--\eqref{discretebilinear25}. 
\end{proof}
\begin{theorem}[Discrete Gårding inequality]\label{discGordinginequality}
Let $\lambda$ be as in definition~\eqref{lambdadfn}, let the conditions of Lemma~\ref{l:suffcond} hold, and let $\gamma \geq 16C_I$, where $C_I$ is the constant in Lemma~\ref{l:inverseineq}.
Then, 
\begin{equation}
\lvert a_{\lambda}(p_h,p_h)\rvert + 2\kappa\|J p_h\|^2_{U}\geq \frac{1}{4}|||\,p_h\,|||^2\qquad \forall p_h\in V_h.
\end{equation}
\end{theorem}

 \begin{proof}
Choosing $q_h = \overline{p}_h$ in expression~\eqref{discretebilinear2} yields
\begin{align}\label{quadform}
a_{\lambda}(p_h,\overline{p}_h)& =  a_{0}(p_h,\overline{p}_h)  
- 2 \int\limits_{\Gamma_{\textrm{I}}}   \bigl(1-\lambda\frac{\zeta}{ \i\kappa} \bigr) \Re\Bigl(\llbracket \overline{p}_h\rrbracket\left\{\frac{\partial p_h}{\partial n}\right\} \Bigr)
% \nonumber\\&\quad    
-\int\limits_{\Gamma_{\textrm{I}}}\frac{\zeta}{ \i\kappa}\bigl(1-\lambda  \frac{\zeta}{ \i\kappa}  \bigr) \biggl|\left\{\frac{\partial p_h}{\partial n}\right\}\biggr|^2 +   \int\limits_{\Gamma_{\textrm{I}}} \lambda\bigl|\llbracket p_h\rrbracket\bigr|^2 .
\end{align} 
From expression~\eqref{quadform}, inequality $\Re z + \Im z \leq\sqrt 2|z|$, and the fact that $\Im a_0(p_h,\overline p_h)\geq0$ follow that
\begin{align}\label{ReIm_est}
&\Re a_{\lambda}(p_h,\overline{p}_h) + \Im a_{\lambda}(p_h,\overline{p}_h) 
\geq  \Re a_{0}(p_h,\overline{p}_h)  
- 2\sqrt2 \int\limits_{\Gamma_{\textrm{I}}}\Bigl|\bigl(1-\lambda\frac{\zeta}{ \i\kappa} \bigr)\Bigr| \biggl|\llbracket \overline{p}_h\rrbracket\left\{\frac{\partial p_h}{\partial n}\right\}\biggr|
 \nonumber\\&\qquad\qquad    
-\sqrt2\int\limits_{\Gamma_{\textrm{I}}}\Bigl|\frac{\zeta}{ \i\kappa}\big(1-\lambda  \frac{\zeta}{ \i\kappa}  \big)\Bigr|\bigg|\!\left\{\frac{\partial p_h}{\partial n}\right\}\!\bigg|^2 +   \int\limits_{\Gamma_{\textrm{I}}}\bigl(\Re\lambda+\Im\lambda\bigr) \bigl|\llbracket p_h\rrbracket\bigr|^2 .
\end{align} 
Consider the second term on the right side of inequality~\eqref{ReIm_est}. 
By identity~\eqref{lambdaidentity1} and inequalities $2ab\leq 2 a^{2}+ b^2/2$ and~\eqref{lambdainequal1},  we obtain
\begin{align}\label{bound1}
 & 2\sqrt2 \int\limits_{\Gamma_{\textrm{I}}}\Bigl|\bigl(1-\lambda\frac{\zeta}{ \i\kappa} \bigr)\Bigr| \biggl|\llbracket \overline{p}_h\rrbracket\left\{\frac{\partial p_h}{\partial n}\right\}\biggr|
  \leq 2\sqrt2 \int\limits_{\Gamma_{\textrm{I}}}\frac{h}{\gamma}\lvert\lambda\rvert \bigg|\llbracket \overline{p}_h\rrbracket\left\{\frac{\partial p_h}{\partial n}\right\} \bigg| \nonumber \\
&\qquad \leq  4\int\limits_{\Gamma_{\textrm{I}}}  \frac{h^2}{\gamma^2}\lvert\lambda\rvert \biggl|\left\{\frac{\partial p_h}{\partial n}\right\} \biggr|^2 
+\frac12\int\limits_{\Gamma_{\textrm{I}}}\lvert\lambda\rvert \big|\llbracket  p_h\rrbracket\big|^2 
%\\&\qquad  
\leq  4\int\limits_{\Gamma_{\textrm{I}}}  \frac{h}{\gamma}\biggl|\left\{\frac{\partial p_h}{\partial n}\right\}\biggr|^2 
+\frac12\int\limits_{\Gamma_{\textrm{I}}}  \lvert\lambda\rvert \bigl|\llbracket  p_h\rrbracket\bigr|^2.
\end{align}
For the third term on the right side of inequality~\eqref{ReIm_est}, using expression~\eqref{lambdainequal2}, we  obtain the bound
\begin{align}\label{bound2}
\sqrt2\int\limits_{\Gamma_{\textrm{I}}}\Bigl|\frac{\zeta}{ \i\kappa}\big(1-\lambda  \frac{\zeta}{ \i\kappa}  \big)\Bigr|\biggl|\left\{\frac{\partial p_h}{\partial n}\right\}\biggr|^2 
\leq 3 \int\limits_{\Gamma_{\textrm{I}}}  \frac{h}{\gamma}\biggl| \left\{\frac{\partial p_h}{\partial n}\right\}\biggr|^2.
\end{align}
Substituting inequalities~\eqref{bound1} and~\eqref{bound2} into expression~\eqref{ReIm_est}, we find that
\begin{align}%\label{discretebilinear5}
&\Re a_{\lambda}(p_h,\overline{p}_h) + \Im a_{\lambda}(p_h,\overline{p}_h)     
\nonumber\\
&\qquad\geq\Re  a_{0}(p_h,\overline{p}_h)  
- \frac{7}{\gamma}\int\limits_{\Gamma_{\textrm{I}}} h\biggl| \left\{\frac{\partial p_h}{\partial n}\right\}\biggr|^2
+   \int\limits_{\Gamma_{\textrm{I}}} \Big(\Re\lambda+\Im\lambda -\frac12\lvert\lambda\rvert\Big)\big|\llbracket { p_h}\rrbracket\big|^2  
\nonumber\\
&\qquad = \Re  a_{0}(p_h,\overline{p}_h)  -  \frac{8}{\gamma}  \int\limits_{\Gamma_{\textrm{I}}}  h \Big| \left\{\frac{\partial p_h}{\partial n}\right\}\Big|^2 +\frac{1}{\gamma}  \int\limits_{\Gamma_{\textrm{I}}}  h \Big| \left\{\frac{\partial p_h}{\partial n}\right\}\Big|^2  
+ \int\limits_{\Gamma_{\textrm{I}}} \Big(\Re\lambda+\Im\lambda -\frac12\lvert\lambda\rvert\Big)\big|\llbracket { p_h}\rrbracket\big|^2.
\end{align} 
By definition~\eqref{bilinear0}, inverse inequality~\eqref{inverseineq}, and since $\gamma\geq16C_I$, we obtain
\begin{align}\label{discretebilinear8}
&\Re a_{\lambda}(p_h,\overline{p}_h) + \Im a_{\lambda}(p_h,\overline{p}_h)  \geq  \frac{1}{2}\int\limits_{\mathclap{\Omega_1\cup\Omega_2}}\lvert\nabla p_h\rvert^2  - \kappa^2  \int\limits_{\mathclap{\Omega_1\cup\Omega_2}} \lvert p_h\rvert^2  
+  \frac{1}{\gamma} \int\limits_{\Gamma_{\textrm{I}}}  h \biggl| \left\{\frac{\partial p_h}{\partial n}\right\}\biggr|^2   
\nonumber\\
&\qquad\qquad\qquad\qquad  +  \Bigl(\frac{1}{2}  -  \frac{8}{\gamma} C_I\Bigr)\int\limits_{\mathclap{\Omega_1\cup\Omega_2}}\lvert\nabla p_h\rvert^2
+ \int\limits_{\Gamma_{\textrm{I}}} \Bigl(\Re\lambda + \Im\lambda -\frac12\bigl|\lambda\bigr|\Bigr)\bigl|\llbracket { p_h}\rrbracket\bigr|^2\nonumber\\
&\qquad\qquad\qquad \geq  \frac{1}{2}\int\limits_{\mathclap{\Omega_1\cup\Omega_2}}\lvert\nabla p_h\rvert^2  - \kappa^2  \int\limits_{\mathclap{\Omega_1\cup\Omega_2}} \lvert p_h\rvert^2  +  \frac{1}{\gamma}  \int\limits_{\Gamma_{\textrm{I}}}  h \biggl| \left\{\frac{\partial p_h}{\partial n}\right\}\biggr|^2   
+ \int\limits_{\Gamma_{\textrm{I}}} \Bigl(\Re\lambda + \Im\lambda -\frac12\lvert\lambda\rvert\Bigr)\bigl|\llbracket { p_h}\rrbracket\bigr|^2.
\end{align} 
Inequalities~\eqref{lambdasplitbound} and~\eqref{lambdaminbound} yields that the last term in expression~\eqref{discretebilinear8} satisfies
\begin{align}\label{jumptermest}
\int\limits_{\Gamma_{\textrm{I}}} \Bigl(\Re\lambda + \Im\lambda -\frac12\lvert\Re\lambda\rvert\Bigr)\bigl|\llbracket { p_h}\rrbracket\bigr|^2
&\geq \frac12\int\limits_{\Gamma_{\textrm{I}}} \lvert\lambda\rvert\big\lvert\llbracket { p_h}\rrbracket\big\rvert^2
- \int\limits_{\Gamma^-_{\textrm{I}}}\Bigl(\frac12\lvert\lambda\rvert + 3\frac{\kappa}{\zetalow}\Bigr)\big\lvert\llbracket { p_h}\rrbracket\big\rvert^2 \nonumber\\
&\geq  \frac12\int\limits_{\Gamma_{\textrm{I}}} \lvert\lambda\rvert\big\lvert\llbracket { p_h}\rrbracket\big\rvert^2
- 4\frac{\kappa}{\zetalow}\int\limits_{\Gamma^-_{\textrm{I}}}\big\lvert\llbracket { p_h}\rrbracket\big\rvert^2.
\end{align}
Substituting inequality~\eqref{jumptermest} into expression~\eqref{discretebilinear8}, using that $2|z|\geq \Re z + \Im z$, and adding a multiple of $\lVert J p_h\rVert^2_U$ we finally obtain
\begin{align}
2\lvert a_\lambda(p_h,\overline p_h)\rvert + 4\kappa\lVert J p_h\rVert^2_U
\geq \frac{1}{2}\int\limits_{\mathclap{\Omega_1\cup\Omega_2}}\lvert\nabla p_h\rvert^2  
+ 3\kappa^2\int\limits_{\mathclap{\Omega_1\cup\Omega_2}} \lvert p_h\rvert^2  
+  \frac{1}{\gamma}  \int\limits_{\Gamma_{\textrm{I}}}h\biggl|\left\{\frac{\partial p_h}{\partial n}\right\}\biggr|^2 
+ \frac12\int\limits_{\Gamma_{\textrm{I}}}\lvert\lambda\rvert\bigl|\llbracket { p_h}\rrbracket\bigr|^2  
\end{align}
from which the conclusion follows.
\end{proof}
 
\subsection{A priori error estimate}

So far, we have considered the following three cases for the interface impedance function $\zeta\in L^\infty(\Gamma_{\text I})$: 
\begin{enumerate}
\item[(i)]
$\zeta\equiv0$, 

\item[(ii)]
$\lvert\zeta\rvert\geq\delta_\zeta > 0$ a.e. on $\Gamma_{\text I}$,

\item[(iii)]
$\lvert\zeta\rvert \geq\delta_\zeta > 0$ a.e. on $\Gamma_{\text I}^-$  (that is, on the subset of $\Gamma_{\text I}$ where $\Im\zeta<0$).
\end{enumerate}

Case (i) is the condition of no interface, case (ii) is the condition that was imposed on the original problem formulation in \S~\ref{Prob_description}.
Our discrete problem~\eqref{discretestateeqn} is constructed to allow the more general case~(iii), for which case~(i) is a special case.

Since discrete problem~\eqref{discretestateeqn} reduces to the standard Nitsche method in case~(i), the a priori analysis is standard and will not be carried out here.
In this section, we will derive an a priori estimate, based on a  Schatz-type argument~\cite{Schatz1974}, \cite[Thm.~(5.7.6)]{brenner2008}, for the finite element approximation in case~(ii). 
The estimate, in turn, implies uniqueness, and thus existence, of solutions to problem~\eqref{discretestateeqn} for $h$ small enough.

We proved discrete stability, in the sense of Theorem~\ref{discGordinginequality}, for case~(iii), but the approach used for a priori error analysis in this section will be restricted to case~(ii).
The reason is the high regularity requirements necessary for the standard  form of the Schatz argument. 
If $\zeta$ vanishes on only a part of $\Gamma_{\text I}$, as is possible in case~(iii), we are in the case of a domain with a cut, for which not enough regularity holds for the proof used in our a priori estimate.

We start by the following estimate that holds in case~(ii).
\begin{lemma}\label{lambdabound2}
Let $\gamma>0$, and assume that  $\zeta\in L^\infty(\Gamma_{\text I})$ satisfies $\lvert\zeta\rvert\geq \delta_{\zeta}>0$ almost everywhere on $\Gamma_{\text I}$.
Then there is a $h_0>0$ such that for each $0<h\leq h_0$, function $\lambda$, as defined in expression~\eqref{lambdadfn}, satisfies
 \begin{equation}\label{e:lambdabound2}
 \lvert\lambda\rvert \leq \frac{\sqrt{2}\kappa}{\delta_{\zeta}}\qquad\text{a.e.\ on $\Gamma_{\text I}$.}
\end{equation}
\end{lemma}
\begin{proof}
Since  $\lvert\zeta\rvert\geq \delta_{\zeta}>0$ almost everywhere on $\Gamma_{\text I}$, Lemma~\ref{l:suffcond} implies that there is an $h_0>0$ such that condition~\eqref{e:zetahcond} is satisfied for each $0<h\leq h_0$.
Expanding the denominator of $\lvert\lambda\rvert^2$, using condition~\eqref{e:zetahcond}, and the lower bound on $\lvert\zeta\rvert$ yields 
\begin{equation}\label{e:ldenomineq2}
\left|\frac{h}\gamma + \frac{\zeta}{\i\kappa}\right|^2
= \frac{h^2}{\gamma^2} + \frac{\lvert\zeta\rvert^2}{\kappa^2} + 2\frac{h\Im\zeta}{\gamma\kappa}
\geq \frac{h^2}{\gamma^2} + \frac{\lvert\zeta\rvert^2}{2\kappa^2} \geq\frac{\delta_\zeta^2}{2\kappa^2},
\end{equation}
from which the conclusion follows.
\end{proof}

Let $\pi^h_i$ be the standard nodal interpolation operator~\cite[Def. 3.3.9]{brenner2008} on the mesh $\mathcal{T}_i^h$ of $\Omega_i$.
We define the interpolation operator $\pi_h:W\to V_h$, where  $W\big|_{\Omega_i}^{\strut} = H^m(\Omega_i)\cap C^0(\overline\Omega_i)$, by  
 \begin{equation}\label{interpoloperator}
 \bigl(\pi_h q\bigr)\big|_{\Omega_i}^{\strut}= \pi^h_i\bigl(q\big|_{\Omega_i}^{\strut}\bigr) \qquad i= 1,2,% \;  \text{ for } q\in W.
\end{equation}
for which the following standard interpolation estimate holds:

\begin{equation}\label{interpolestimate}
|||\, p-\pi_hp\, |||\leq  C_{\kappa\gamma}   h^{s}\|p\|_{H^{s+1}(\Omega_1\cup\Omega_2)},
\end{equation}
where $C_{\kappa\gamma}$ is a constant that depends on $\kappa$ and $\gamma$, $s\in[0,k]$, and where $k$ is the maximal polynomial degree of the elements in $V_h$.
This estimate relies on trace inequality~\eqref{trace2} together with scalings to and from a reference element, analogously as for discontinuous Galerkin methods~\cite[\S~2]{Arnold1982}.
In three space dimensions, we need to use that the family of surface meshes on $\Gamma_{\text I}$ is non degenerate if the volume mesh family is non degenerate, as discussed in remark~\ref{r_nondegen}.

In addition to a regularity condition on the solution to variational problem~\eqref{ContVarationalproblem}, our proof of Theorem~\eqref{errorestimateThm} requires the following  regularity assumption for a particular dual problem associated with the discrete norm~\eqref{triplenorm}.
\begin{assumption}\label{dualreg}
For any pair of functions $f_1\in L^2(\Omega_1\cup\Omega_2)$, $f_2\in H^{1/2}(\Gamma_\text{I}^-)$, the solution to the variational problem of finding $w\in H^1(\Omega_1\cup\Omega_2)$ such that 
\begin{equation}
a(q, w) = \kappa\int\limits_{\mathclap{\Omega_1\cup\Omega_2}} f_1q 
+ \frac{1}{\delta_\zeta}\int\limits_{\mathclap{\Gamma_\text{I}^-}}f_2\jump q
\qquad\forall q\in H^1(\Omega_1\cup\Omega_2)
\end{equation}
satisfies the condition
\begin{equation}
\|w\|_{H^2(\Omega_1\cup\Omega_2)} \leq C_r\big(\|f_1\|_{L^{2}(\Omega_1\cup\Omega_2)}+\|f_2\|_{H^{1/2}(\Gamma_I)}\big),
\end{equation}
where $C_r$ depends on $\kappa$ and $\delta_\zeta$ but not on $f_1$ and $f_2$.
\end{assumption}

\begin{theorem}\label{errorestimateThm}
Let $p$ be the solution of problem~\eqref{ContVarationalproblem}, in which the impedance function $\zeta\in L^\infty(\Gamma)$ is assumed to satisfy $\lvert\zeta\rvert \geq\delta_\zeta >0$ a.e.\ on $\Gamma_\text{I}$, and let $p_h$ be a solution to corresponding discrete problem~\eqref{discretestateeqn},  
Provided  that $p$ satisfies the regularity condition $p\in H^{1+\sigma}(\Omega_1\cup\Omega_2)$ for some $\sigma\geq 1$ and that Assumption~\ref{dualreg} holds, there exists a mesh size limit $h_1>0 $ and a constant $C$ such that for each $0<h\leq h_1$,
\begin{equation}\label{errorestimate}
\begin{aligned}
%||| e |||        \leq Ch^{s}\|p\|_{H^{s+1}(\Omega_1\cup\Omega_2)}.
|||\, p-p_h\, ||| \leq Ch^{s}\|p\|_{H^{s+1}(\Omega_1\cup\Omega_2)}.
\end{aligned}
\end{equation} 
where $s=\min\{\sigma, k\}$ and where $k$ is the maximal polynomial order of the elements in $V_h$.

\end{theorem}
\begin{proof}
Define
\begin{align}
   e &= p - p_h \label{edef}, \\
   e_h &= \pi_h e =  \pi_hp-p_h \label{ehdef}.
\end{align}

First we estimate $|||\, e_h\, |||$ using the continuity of $a_\lambda$ and the Gårding inequality.
Let $w\in H^1(\Omega_1\cup\Omega_2)$ be the solution of the dual problem 
\begin{equation}\label{dualproblem} 
 a(q,w) =  \kappa  \int\limits_{\mathclap{\Omega_1\cup\Omega_2}} {q}\overline{e}_h+ \frac{1}{\delta_{\zeta}}  \int\limits_{\Gamma^-_{\text{I}}}\jump q\llbracket \overline{e}_h \rrbracket\qquad \forall q\in H^1(\Omega_1\cup\Omega_2), %f_g(q)
\end{equation}
where $a$ is the bilinear form defined in expression~\eqref{contbilinear_a}. 
Since $a_\lambda$ is symmetric and consistent (by Lemma~\ref{consistentlemma}), $w$ satisfies 
\begin{equation}\label{discrdualproblem}
 a_\lambda(q_h,w) =  \kappa  \int\limits_{\mathclap{\Omega_1\cup\Omega_2}} {q_h}\overline{e}_h+  \frac{1}{\delta_{\zeta}} \int\limits_{\Gamma^-_{\text{I}}} \llbracket q_h\rrbracket\llbracket \overline{e}_h \rrbracket, \;\; \forall q_h\in V_h.
\end{equation}

Due to Assumption~\eqref{dualreg}, there is a constant $C_r$ such that 
\begin{equation}\label{H2regularity}
\|w\|_{H^2(\Omega_1\cup\Omega_2)} \leq C_r\big(\|e_h\|_{L^{2}(\Omega_1\cup\Omega_2)}+\|e_h\|_{H^{1/2}(\Gamma_I)}\big).
\end{equation}
Moreover, by trace inequality~\eqref{trace1} for $m=0$, we can bound the $H^2$ norm of $w$ by 
\begin{equation}
  \label{eq:boundH2tripple}
  \|w\|_{H^2(\Omega_1\cup\Omega_2)} \leq C_r \big(\|e_h\|_{L^{2}(\Omega_1\cup\Omega_2)}+c_1 \|e_h\|_{H^{1}(\Omega_1\cup\Omega_2)}\big)
  \leq C_t ||| e_h |||,
\end{equation}
for $C_t =C_r(1+c_1)$.

Trace inequality~\eqref{trace1}, for $m=0$ as well as $m=1$, and Lemma~\ref{lambdabound2} yield 
that there is an $h_0>0$ such that for each $0<h<h_0$,  the  estimate 
\begin{equation}\label{H2bound}
|||w||| \leq C_b\|w\|_{H^2(\Omega_1\cup\Omega_2)}
\end{equation}
holds, where $C_b$ depends on $\kappa$ , $\gamma$, $h_0$ and $\delta_{\zeta}$.

By choosing $q_h=e_h$ in equation~\eqref{discrdualproblem}, recalling definition~\eqref{Unorm}, and using orthogonality condition~\eqref{orthogocond}, we find that 
\begin{align}
 \|J e_h\|^2_{U} &= \kappa  \int\limits_{\mathclap{\Omega_1\cup\Omega_2}} |e_h|^2 +\frac{1}{\delta_{\zeta}} \int\limits_{\Gamma_{\text{I}}}  \big|\llbracket { e_h}\rrbracket\big|^2  \nonumber \\ 
  &=  
 a_{\lambda}(e_h,w) =
 a_{\lambda}(e_h-e,w) + a_{\lambda}(e,w) \nonumber \\ 
  &= 
  a_{\lambda}(e_h-e,w) + a_{\lambda}(e,w-\pi_h w)  \nonumber \\ 
  &=  
 a_{\lambda}(e_h-e,w) + a_{\lambda}(e-e_h,w-\pi_h w) + a_{\lambda}(e_h,w-\pi_h w).
  \label{eq:tnormehbound}
\end{align}

The continuity of $a_\lambda$ (Theorem~\ref{discContinequality}) implies that
\begin{equation}
  \|Je_h\|^2_{U} \le
  C_c||| e_h-e||| \big( |||w||| + |||w-\pi_h w||| \big) + 
  C_c|||e_h||| \, |||w-\pi_h w|||.
  \label{eq:firstehubound}
\end{equation}
Inequality~\eqref{H2bound} and interpolation estimate~\eqref{interpolestimate} with $s=1$ yields that
\begin{equation}
  \|J e_h\|^2_{U} \le
  \big( ||| e_h-e||| (C_b  +C_{\kappa\gamma} h ) + |||e_h||| \, C_{\kappa\gamma} h \big)  C_c \|w\|_{H^2({\Omega_1\cup\Omega_2})}.
  \label{eq:firstehubound}
\end{equation}
Then, by inequality~\eqref{eq:boundH2tripple}, we obtain the bound 
\begin{equation}
  \|J e_h\|^2_{U} \le
  \big( ||| e_h-e||| (C_b  +C_{\kappa\gamma} h ) + |||e_h||| \, C_{\kappa\gamma} h \big)  C_c C_t ||| e_h |||.
  \label{eq:secondehubound}
\end{equation}

From the discrete Gårding inequality (Theorem~\ref{discGordinginequality}),  Galerkin orthogonality~\eqref{orthogocond}, and the continuity of $a_\lambda$ (Theorem~\ref{discContinequality}), we find that
\begin{align}
  \frac14 |||e_h|||^2 &\le  
  \lvert a_{\lambda}(e_h,e_h)\rvert + 2\kappa \|J e_h\|^2_{U} \nonumber \\ 
  &= 
  |a_{\lambda}(e_h-e,e_h)| + 2\kappa \| Je_h\|^2_{U} \le
  C_c||| e_h-e||| \, ||| e_h||| + 2\kappa \| J e_h\|^2_{U}.
  \label{eq:tnormehbound}
\end{align}
By combining the bounds~\eqref{eq:secondehubound} and~\eqref{eq:tnormehbound} and rearranging the terms, we find that 
\begin{equation}
  \left( \frac14 - 2\kappa C_{\kappa\gamma}  C_c C_t h\right) |||e_h|||^2 \le  
  C_c\left(1 + 2\kappa (C_b + C_{\kappa\gamma} h) C_t  \right) ||| e_h-e||| \, ||| e_h|||.
\end{equation}
Then, provided that $h\le  \min(h_0, h_1)$, where $h_1 = 1/(16 \kappa C_{\kappa\gamma} C_c C_t)$, we have
\begin{equation}%\label{eh_bound}
  |||e_h||| \le C_{eh} ||| e_h-e|||,
  \label{eq:tnormehbound2}
\end{equation}
where 
$C_{eh}=  1 + 8C_c (1+2\kappa C_tC_b)$.
By the triangle inequality and inequality~\eqref{eq:tnormehbound2}, we get
%to bound triple norm of $e$ as
\begin{equation}\label{eq:tnormtriangineq}
   |||e||| \le |||e-e_h||| + |||e_h||| \le ( 1+ C_{eh} )||| e_h-e|||. %\qedhere
\end{equation}
Finally, by inequality~\eqref{eq:tnormtriangineq}, definitions~\eqref{edef}  and~\eqref{ehdef}, and interpolation estimate~\eqref{interpolestimate}, we obtain
%to bound triple norm of $e$ as
\begin{equation}
   |||p-p_h|||  \le ( 1+ C_{eh} )|||p-\pi_hp|||  \le C h^{s}\|p\|_{H^{s+1}(\Omega_1\cup\Omega_2)}, %\qedhere
\end{equation}
where $C=C_{\kappa\gamma} ( 1+ C_{eh} ).$
\end{proof}

Note that Theorem~\ref{errorestimateThm} assumes existence of at least one $p_h$ that solves the discrete problem~\eqref{discretestateeqn}.
However, it is only in the case of $\ell = 0$ that existence can be assumed a priori;
in this case, we know that at least the trivial solution $p_h=0$ satisfies equation~\eqref{discretestateeqn}.
Since corresponding solution $p$ to problem~\eqref{ContVarationalproblem}  vanishes due to uniqueness  (Lemma~\ref{contuniquenesslemma}), Theorem~\ref{errorestimateThm} implies that the trivial solution is the only solution also to the homogeneous discrete problem~\eqref{discretestateeqn} for each $0<h\leq h_1$, which in turn implies uniqueness of solutions to the inhomogeneous discrete problems.
Since uniqueness implies existence for finite-dimensional linear systems, it thus exists a unique solution to the discrete problem~\eqref{discretestateeqn} for each $\ell$ under the assumptions of Theorem~\ref{errorestimateThm}.
  
\section{Numerical Experiments}\label{s:NumResults}

\subsection{Convergence test}\label{s:Convergencetest}

As a first experiment, we study the convergence properties of the proposed method for a test case involving planar wave propagation in a  two-dimensional strip and compare the results with the ones obtained using  a standard finite element implementation based on variational form~\eqref{ContVarationalproblem}.
We consider the boundary value problem~\eqref{stateequation} in the domain $\Omega_1\cup\Omega_2=\{x\in\mathbb{R}^2: 0<|x_1|<1, 0<x_2<0.1\}$ with an interface boundary $\Gamma_{\text{I}} =\{(0,x_2)\in\mathbb{R}^2: 0<x_2<0.1\}$.
In other words, $\Omega_1\cup\Omega_2$ is a waveguide of length 2~m and width 0.1~m, and the interface is placed vertically at the middle of the waveguide.
Thus, domain $\Omega_1\cup\Omega_2$ is topologically equivalent to the setup in Figure~\ref{fig:DomCases}~(a).

Since $\Omega_1\cup\Omega_2$ is a narrow wave guide, the exact solution is given by
\begin{equation}
p(x_1,x_2)= \begin{cases}
  e^{-\text{i}\kappa(x_1+1)}+\frac{\zeta}{2+\zeta}e^{\text{i}\kappa(x_1-1)}
& -1<x_1<0,\, 0<x_2<0.1
\\[5pt] 
\frac{2}{2+\zeta} e^{-\text{i}\kappa(x_1+1)} %+b_r e^{\text{i}\kappa(x_1-1)}  
& \phantom{-}0<x_1<1,\, 0<x_2<0.1.
\end{cases} %\left{ \right.
\end{equation}
 \begin{figure}
 	\centering
 	\subfloat{\includegraphics[width=50mm]{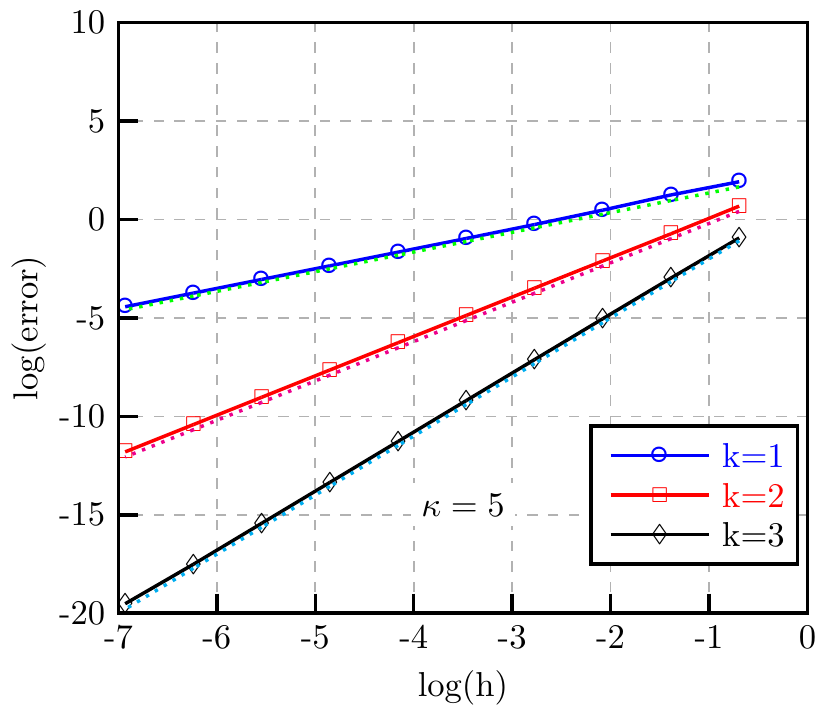}\label{fig:kappa5}}
 	\subfloat {\includegraphics[width=50mm]{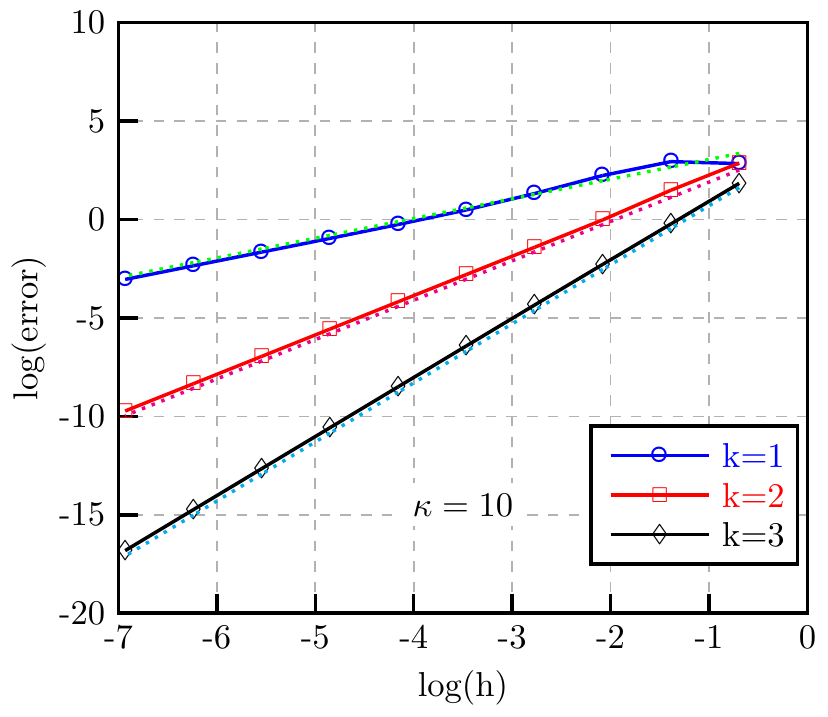}\label{fig:kappa10}}
 	\\[-10pt]
 	\subfloat{\includegraphics[width=50mm]{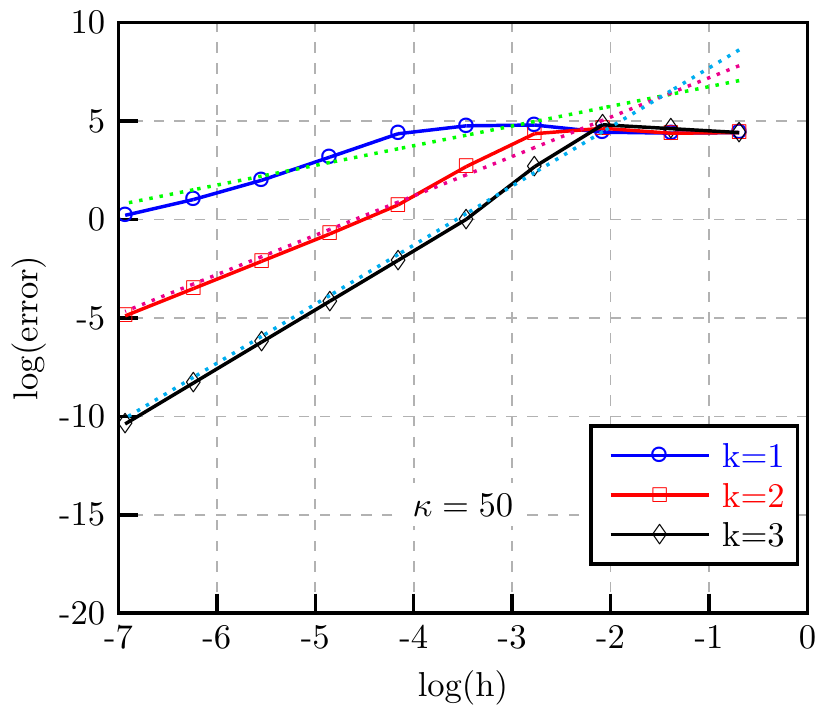}\label{fig:kappa50}}
 	\subfloat {\includegraphics[width=50mm]{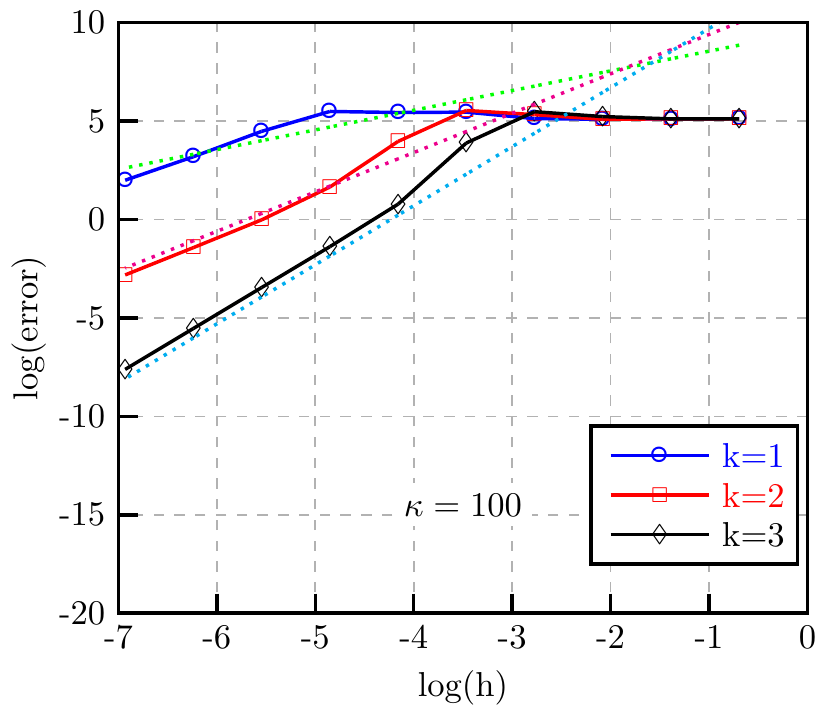}\label{fig:kappa100}}
 	\\[-10pt]
 	\subfloat{\includegraphics[width=50mm]{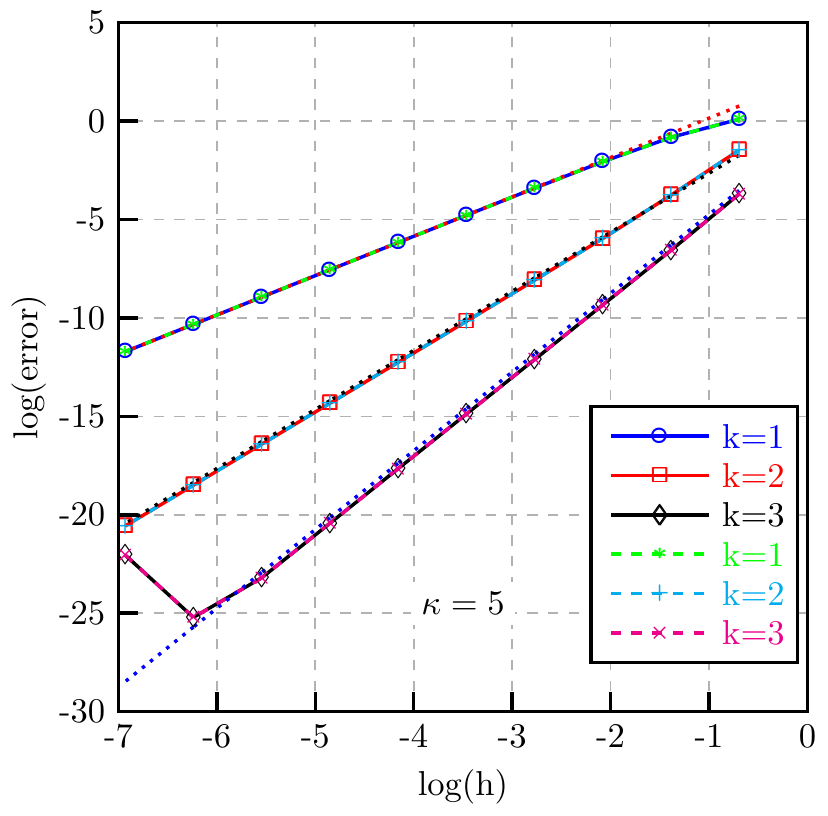}\label{fig:kappa5}}
 	\subfloat {\includegraphics[width=50mm]{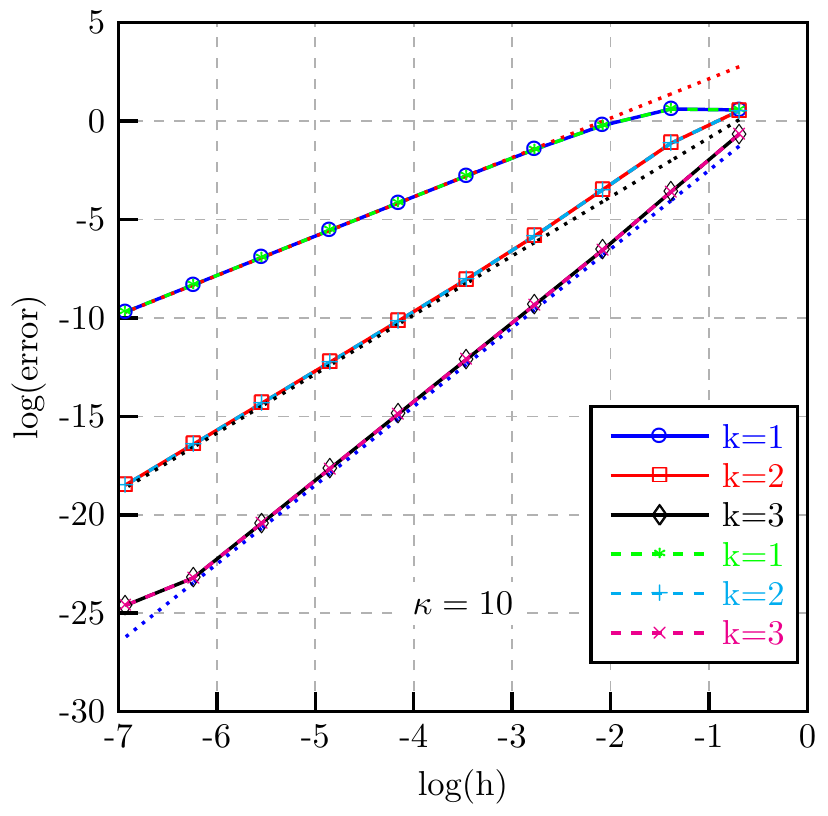}\label{fig:kappa10}}
 	\\[-10pt]
 	\subfloat{\includegraphics[width=50mm]{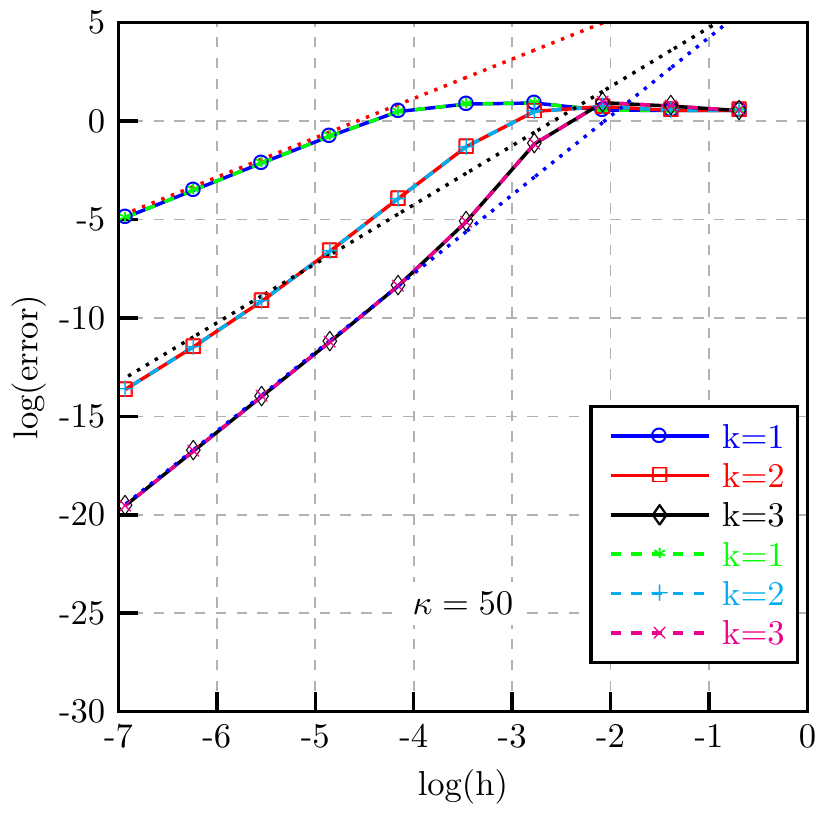}\label{fig:kappa50}}
 	\subfloat {\includegraphics[width=50mm]{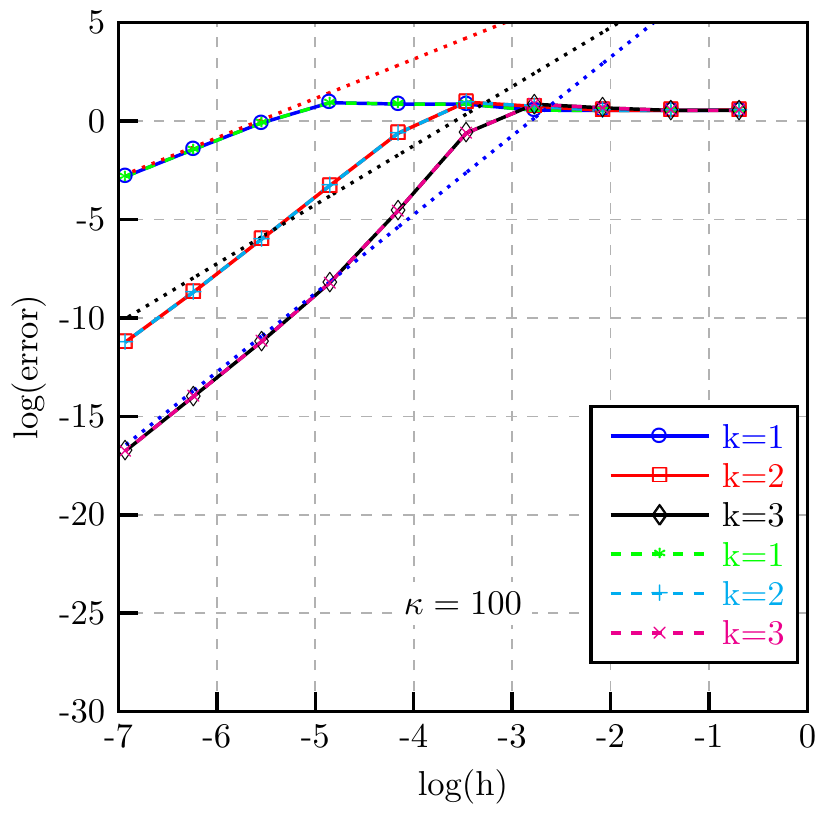}\label{fig:kappa100}}
 	\caption{Convergence of the proposed method for the interface problem with acoustic impedance $\zeta = 0.21+0.10\mathrm{i}$ for polynomial order $k\in\{1, 2, 3\}$ and wave numbers 
 		$\kappa = 5$, $\kappa = 10$, $\kappa = 50$, and $\kappa = 100$.  
 		The top four pictures show the error measured in the mesh-dependent norm~\eqref{triplenorm}, whereas the bottom four show the erros in the $L^2(\Omega)$ norm.
 		The dotted lines are references for convergence of order 1--4.}\label{fig:convrates1}
 \end{figure}

 \begin{figure}
 	\centering
 	\setlength{\unitlength}{0.5\textwidth}
 	\subfloat{\includegraphics[width=45mm]{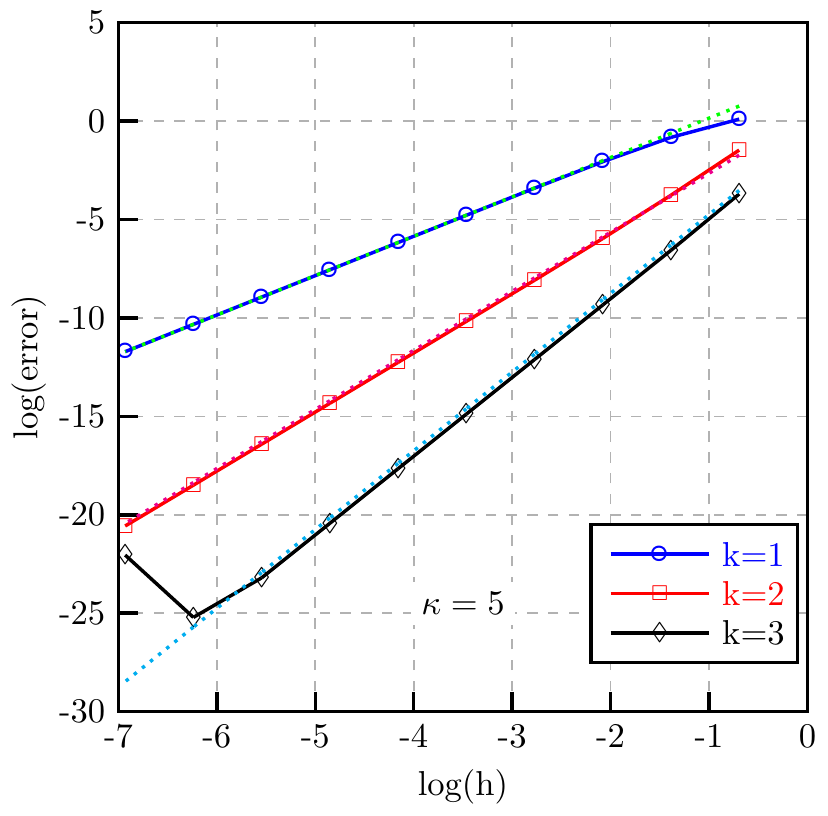}\label{fig:kappa5}}
 	\subfloat {\includegraphics[width=45mm]{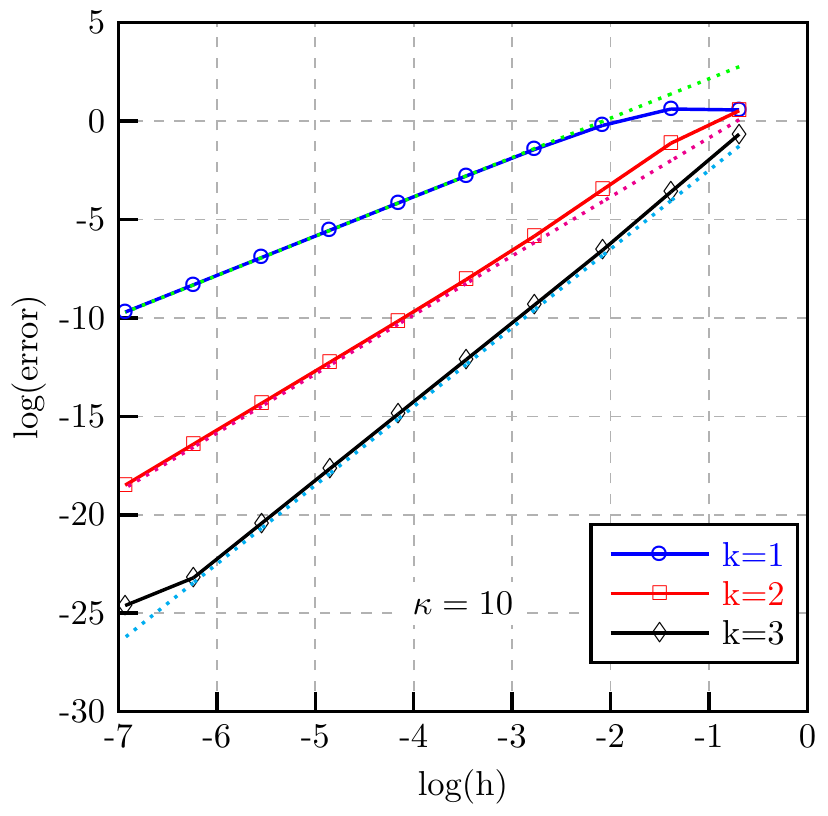}\label{fig:kappa10}}
 	\\[-10pt]
 	\subfloat{\includegraphics[width=45mm]{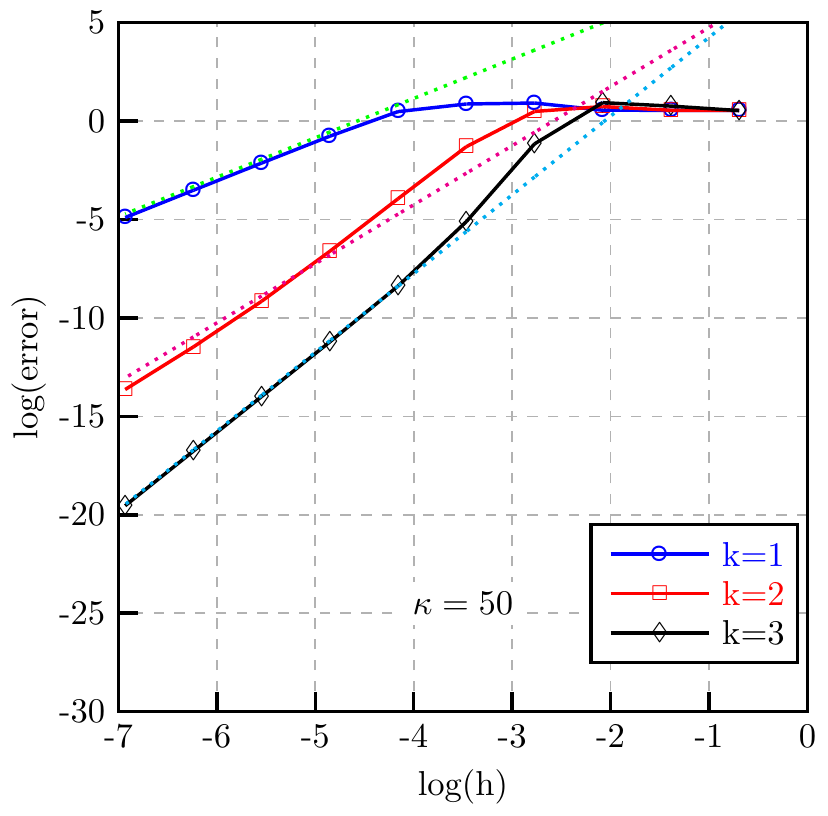}\label{fig:kappa50}}
 	\subfloat {\includegraphics[width=45mm]{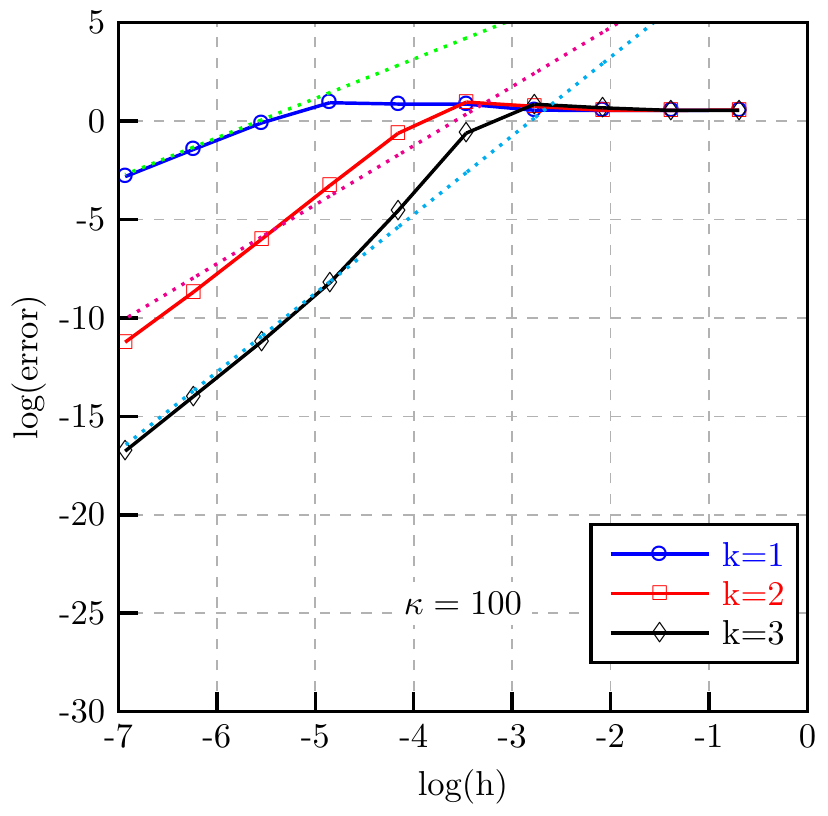}\label{fig:kappa100}}
 	\caption{Convergence rates of the proposed method for the interface problem with acoustic impedance $\zeta = 0$ for $k\in\{1, 2, 3\}$ and
 		$\kappa = 5$, $\kappa = 10$, $\kappa = 50$, and $\kappa = 100$.
 		The dotted lines are reference lines for second, third, and fourth order $L^2$-convergence.}
 	\label{fig:convratezeta2}
 \end{figure}

The computational domain is discretized by square elements.
The proposed as well as the standard finite element methods are implemented in Matlab for piecewise bilinear ($k =1$), biquadratic ($k =2$), and bicubic ($k =3$) finite element spaces.
To study the effect of wave number $\kappa$ on the convergence rate, we consider the four different wave numbers $\kappa=5$, $10$, $50$, and $100$.

Figure~\ref{fig:convrates1} shows the convergence behavior for $\zeta=0.21+0.10\i$. 
The top four pictures show the errors measured in the mesh-dependent norm~\eqref{triplenorm}; lines with circle, square, and diamond  indicate the behavior for polynomial order $k =1$, $k =2$, and $k =3$, respectively.
We conclude that  asymptotic convergence rates agree with the optimal rates established in Theorem~\ref{errorestimateThm}.
The bottom four pictures of figure~\ref{fig:convrates1} show corresponding convergence behavior  in the $L^2(\Omega)$ norm. 
Also displayed in the four bottom pictures (dashed lines marked with asterisk, plus, and x marks) is the convergence behavior for the standard finite element  method based on a discretization of variational form~\eqref{ContVarationalproblem}.
Note that the curves for the standard and new methods are on top of each other.
We conclude that the convergence rates are optimal also in the $L^2(\Omega)$ norm for this test case and that the proposed method behaves as the standard finite element method for $\zeta$ not close to zero.

Figure~\ref{fig:convratezeta2} shows the performance of the proposed method for $\zeta=0$.
In this case,  the exact solution is continuous across the interface, and the  finite element method based on variational form~\eqref{ContVarationalproblem} is not applicable. 
The lines with circle, square, and diamond marks show second, third, and fourth order $L^2(\Omega)$-convergence for $k =1$, $k =2$, and $k =3$, respectively, which verifies optimal convergence of the proposed method also in the limit case $\zeta=0$.

Figures~\ref{fig:convrates1} and \ref{fig:convratezeta2} suggest asymptotically optimal convergence rate of the proposed method independent of acoustic impedances $\zeta$.
However, as expected, the error increases significantly with increasing wave number due to the so-called pollution error of standard continuous Galerkin methods~\cite{IhBa95}. We also note that higher-order methods are particularly effective to reduce the error for higher wave numbers.

%\clearpage

\subsection{Examples in 2D}\label{examples}

\subsubsection{Without surface waves}\label{Pressure_jump}
In this example $\Omega_1\cup\Omega_2\in\mathbb{R}^2$ is an arbitrary cross-section of a simple cylindrical reactive muffler; that is,  
$\Omega_1\cup\Omega_2\in\mathbb{R}^2$ is composed of two rectangular domains $\Omega_1$ and $\Omega_2$, as depicted in Figure~\ref{fig:DomCases} (b), and the equations are solved in cylindrical coordinates, assuming rotational symmetry around an axis placed at the lowest boundary.
Domain $\Omega_1$ has length 0.9~m and width 0.05~m, and $\Omega_2$ has length 0.5~m and width 0.05~m.
We solve boundary value problem~\eqref{stateequation} using our proposed method for two interface conditions on $\Gamma_{\text{I}}$, $\zeta= 1+1\mathrm{i}$ and $\zeta= 0$.
On the left boundary $\Gamma_{\text{io}}$, we set $g=1$ with $\kappa= 23.8$.
This condition imposes an incoming plane wave of unit amplitude and absorbs outgoing plane waves.
On the right $\Gamma_{\text{io}}$, we set $g=0$; that is, no wave is entering and the outgoing plane wave is absorbed.
This example is implemented in Comsol Multiphysics using the software's ``weak form'' facility, where all integrals associated with the variational form~\eqref{discretestateeqn} can be specified symbolically.
The finite element discretization uses a uniform mesh with square quadratic elements of side length $h = 3.96\times10^{-4}~\text{m}$.

Figure~\ref{fig:Pressurejump} shows pressure field distributions. 
Plots $(a)$ and $(b)$ on the left display the  real part of the pressure for interface impedance $\zeta= 1+1\mathrm{i}$.
The solution exhibits a pressure jump across the lossy interface boundary $\Gamma_{\text{I}}$.
The plots on the right shows the real part of the pressure for  $\zeta= 0$.
Figures~\ref{fig:Pressurejump} $(c)$ and $(d)$ show the continuity of the solution across $\Gamma_{\text{I}}$ as expected according to expression~\eqref{e:Zdef}.
The proposed method is thus capable to handle interface conditions with $\zeta= 0$ and $\zeta\neq 0$ without problem.
\begin{figure}
\centering
 \setlength{\unitlength}{1\textwidth}
\subfloat{\includegraphics[width=74mm]{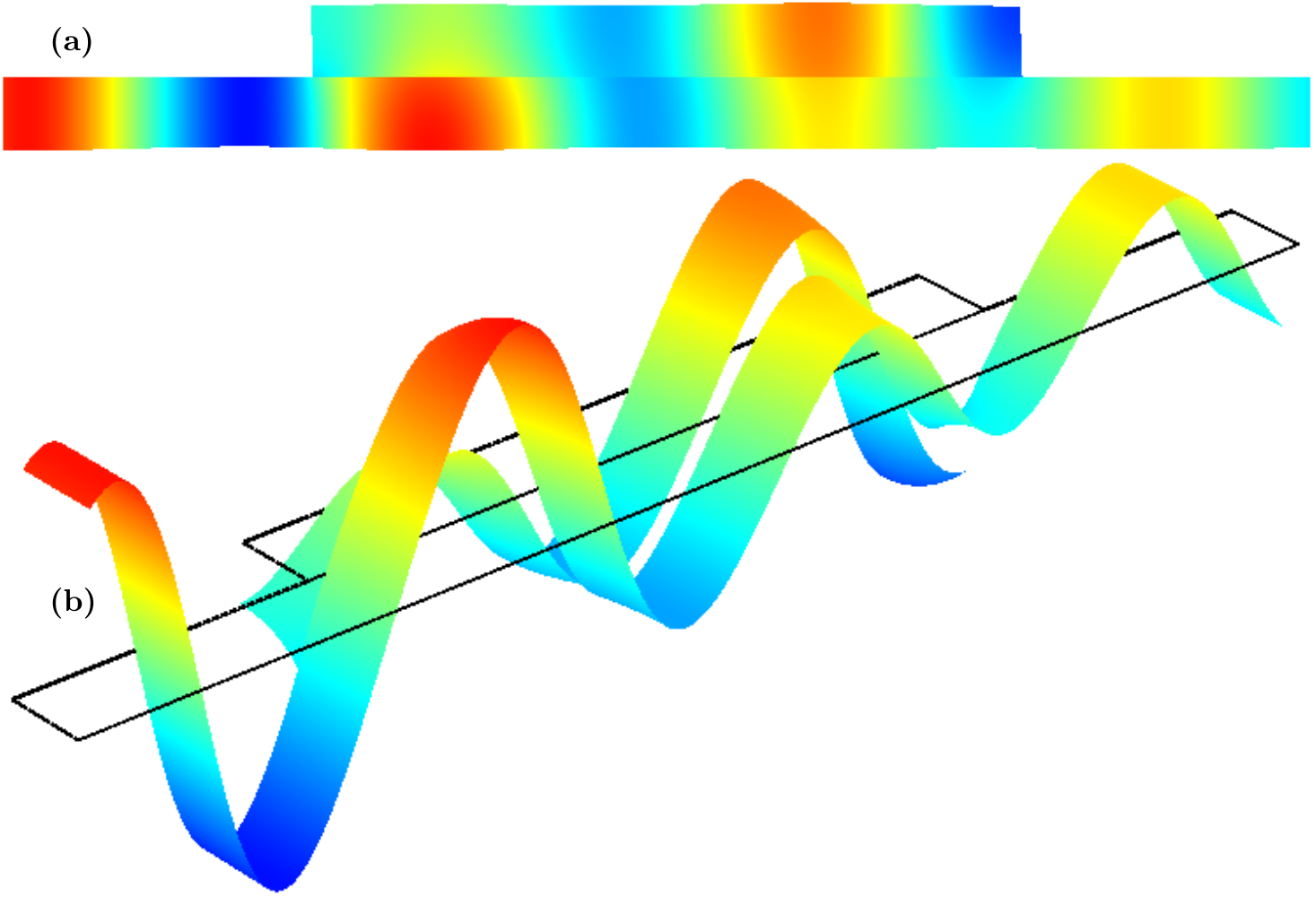}\label{fig:Muffler2D}} \put(-0.25,0.05){$ \zeta = 1+1\text{i}$} \hspace{1cm}  
\subfloat{\includegraphics[width=70mm]{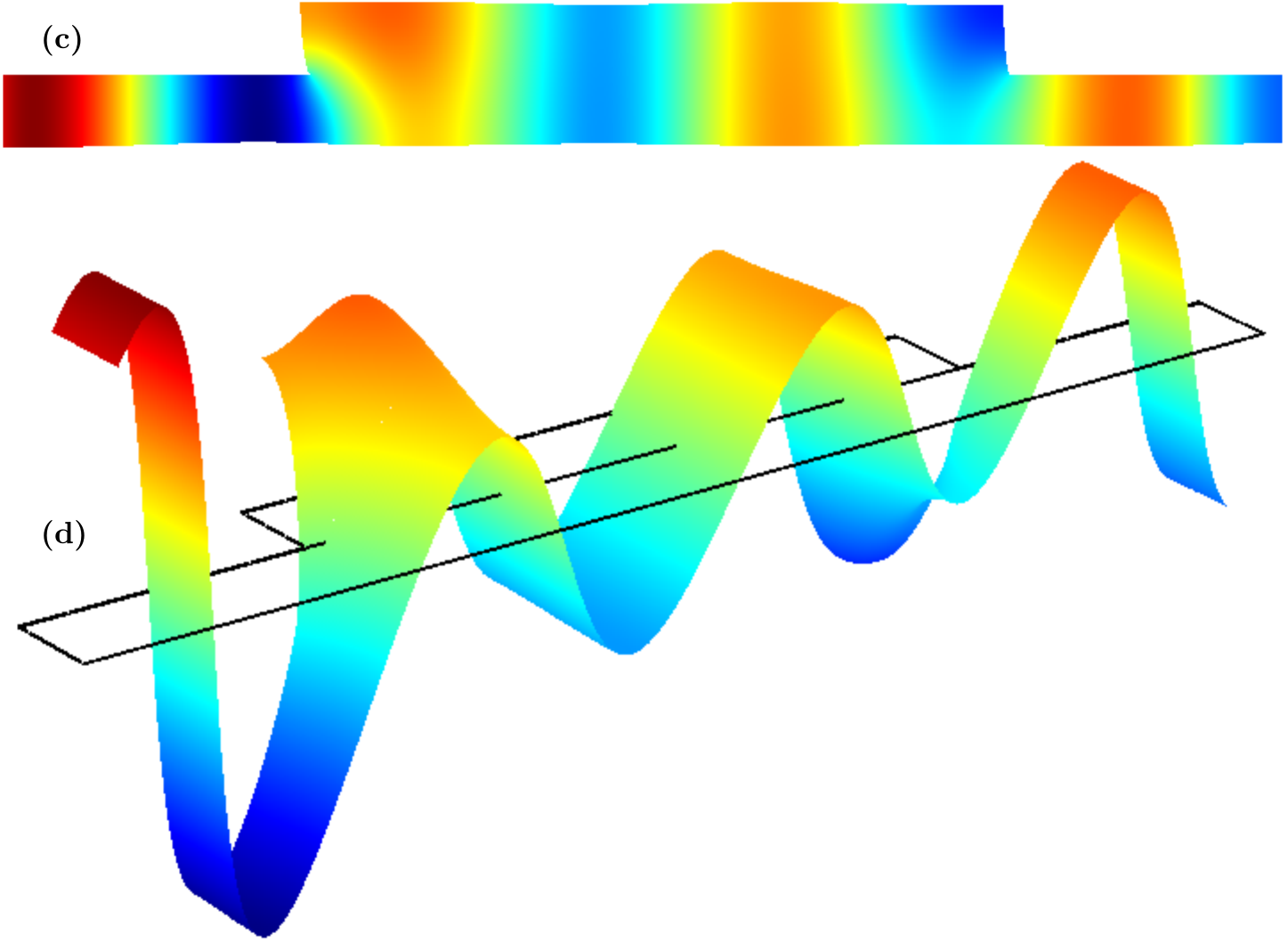}\label{fig:Muffler2D}}\put(-0.24,0.05){$\zeta = 0$}  
 \caption{ Solutions from the proposed method for the interface problem with $\zeta\neq 0$ and $\zeta= 0$.  }
 \label{fig:Pressurejump}
\end{figure}

\subsubsection{With surface wave}
%In this example we present surface waves. 
We consider the same problem setup as in Section~\ref{Pressure_jump} except for the interface conditions and wave numbers.
As discussed in Section~\ref{LinearAcoustics}, interface impedances with a negative imaginary part may produce surface waves in a layer of depth $\delta = O(|\Im \zeta|/\kappa)$ around the surface.
The wave number of these waves increases with decreasing $|\Im \zeta|$, and approaches $O(1/\delta)$ as $\Im \zeta\to0$.  

To observe an unattenuated surface wave we consider a purely imaginary acoustic impedance, $\zeta$.
Figure~\ref{fig:surfacewave} shows the behaviour of the surface waves for varying bulk wave number $\kappa$  and a fixed impedance $\zeta = -0.2\text{i}$, whereas 
Figure~\ref{fig:surfacewave2} shows the behaviour for a fixed wave number  $\kappa$  and a varying impedance.
In both cases, we observe the predicted scaling of the layer depth and local surface wave number. 

 \begin{figure}
\centering
 \setlength{\unitlength}{0.5\textwidth}
  \subfloat{\includegraphics[width=72mm]{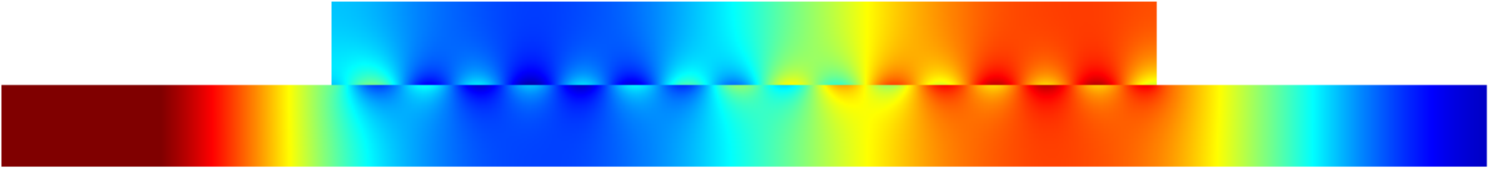}\label{fig:SurfaceWave2}}\put(-0.55,0.12){$\kappa=10$} \hspace{1cm}
 \subfloat{\includegraphics[width=72mm]{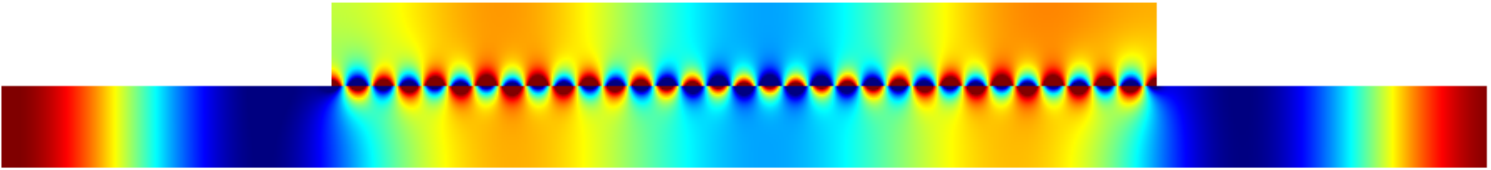}\label{fig:SurfaceWave3}}  \put(-0.55,0.12){$\kappa=20$}\\ \vspace{.35cm}
 \subfloat{\includegraphics[width=72mm]{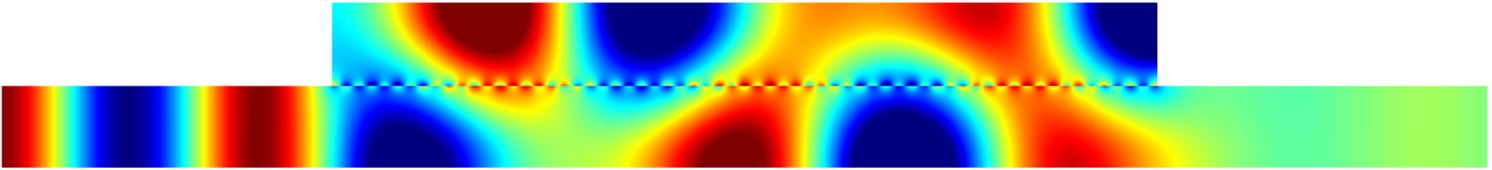}\label{fig:SurfaceWave3}}  \put(-0.55,0.12){$\kappa=40$} \hspace{1cm}
 \subfloat{\includegraphics[width=72mm]{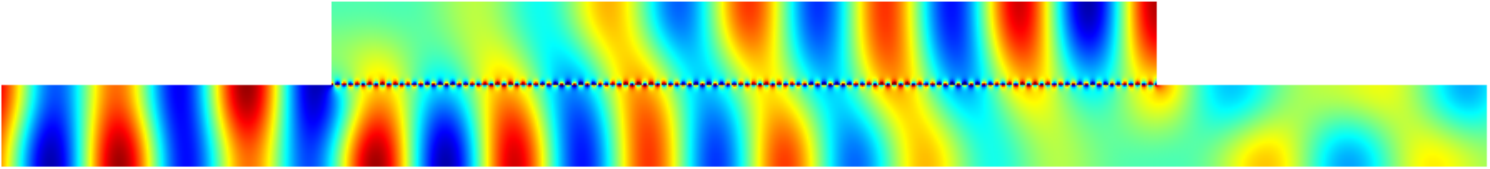}\label{fig:SurfaceWave3}}  \put(-0.55,0.12){$\kappa=80$}
 \caption{ Surface waves corresponding to different wave numbers for $\zeta=-0.2\text{i}$.}
 \label{fig:surfacewave}
\end{figure}

\begin{figure}
\centering
 \setlength{\unitlength}{0.5\textwidth}
 \subfloat{\includegraphics[width=72mm]{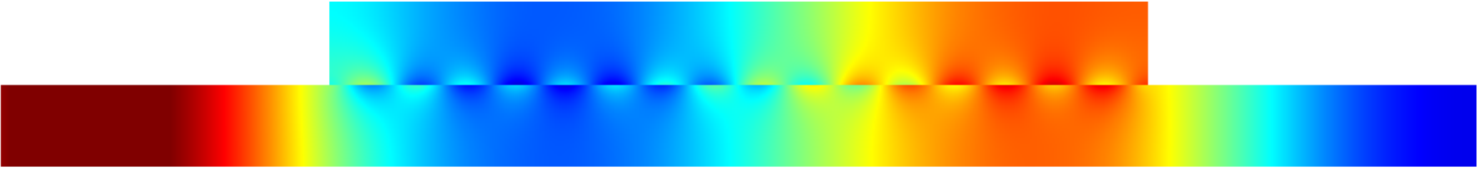}\label{fig:SurfaceWave2}}\put(-0.55,0.12){$\zeta=-0.2\text{i}$} \hspace{1cm}
 \subfloat{\includegraphics[width=72mm]{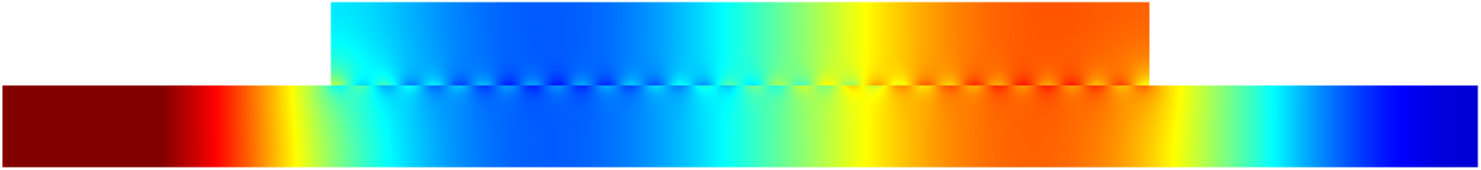}\label{fig:SurfaceWave3}}  \put(-0.55,0.12){$\zeta=-0.1\text{i}$}\\ \vspace{.35cm}
 \subfloat{\includegraphics[width=72mm]{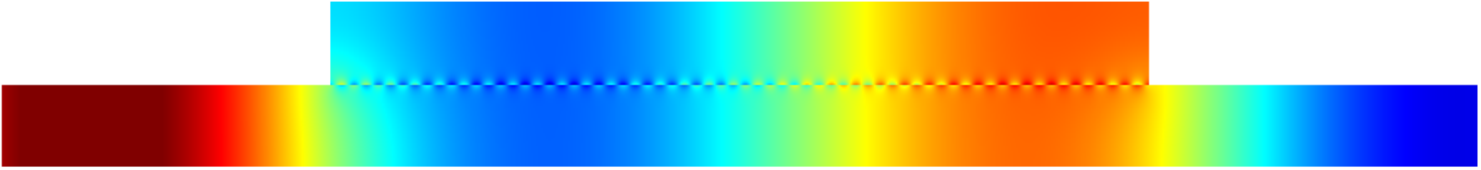}\label{fig:SurfaceWave3}}  \put(-0.55,0.12){$\zeta=-0.05\text{i}$} \hspace{1cm}
 \subfloat{\includegraphics[width=72mm]{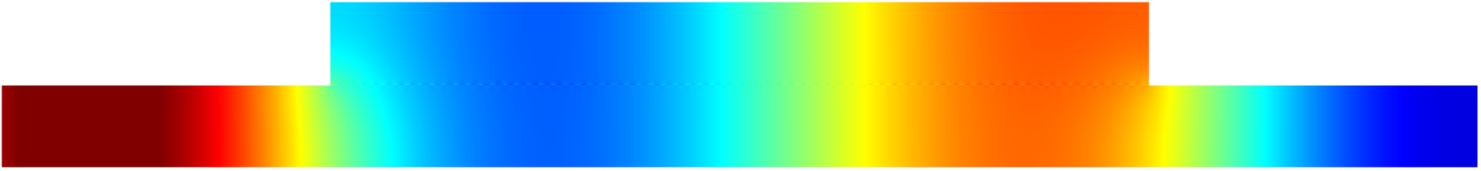}\label{fig:SurfaceWave3}}  \put(-0.55,0.12){$\zeta=-0.025\text{i}$}
 \caption{ Surface waves behaviour as  $\text{Im}(\zeta)$ decreases for $\kappa=10.5$.}
 \label{fig:surfacewave2}
\end{figure}

\subsection{Example in 3D }
This final example shows how the proposed method is capable to handle more complicated domains with multiple interface boundaries.
Here $\cup_{l=1}^4\Omega_l\subset\mathbb{R}^3$ resembles a typical two-chambered reactive muffler with inlet and outlet pipes that extend both outside and inside the chambers. 
As illustrated in Figure~\ref{fig:Exa3CompDom}, the interior of the muffler has two chambers, denoted  $\Omega_2$ and $\Omega_3$, that are separated by an interface boundary $\Gamma_{\textrm{3}}$.
The inlet and outlet pipes, denoted by $\Omega_1$ and $\Omega_4$, extend into chambers $\Omega_2$ and $\Omega_3$, respectively.
The inlet opening of $\Omega_1$ and the outlet opening of $\Omega_4$ are denoted by  $\Gamma_{\textrm{io}}$, and the other openings of $\Omega_1$ and $\Omega_4$ are labeled $\Gamma_{2}$ and $\Gamma_{4}$, respectively.
The part of the inlet and outlet pipes that extend into the two chambers are often perforated with holes much smaller than the operational wavelength for the muffler, which means that we can model these surfaces using a transmission impedance.
These perforated boundaries are here denoted $\Gamma_{1}$ and $\Gamma_{5}$.
All other boundaries, assumed to be sound hard, are denoted by $\Gamma_{\textrm{s}}$.

 \begin{figure}
\centering
 \setlength{\unitlength}{0.5\textwidth}
\subfloat{\includegraphics[width=78mm]{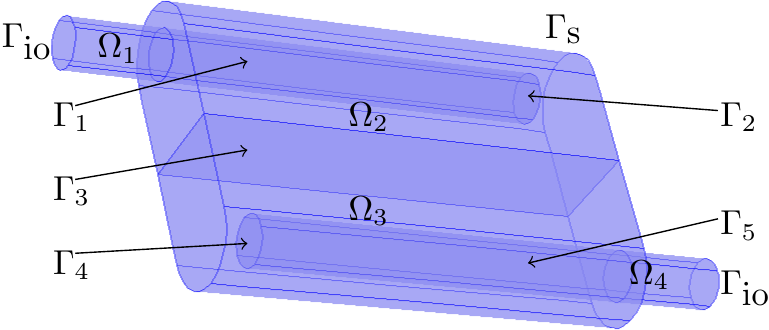}}  
 \caption{ A computational domain in $\mathbb{R}^3$ that resembles a typical two chambered reactive muffler with extended offset inlet and outlet pipes. }
 \label{fig:Exa3CompDom}
\end{figure}

Hence, in this example,  the computational domain is the union of the disjoint, open, and connected domains $\Omega_1$, $\Omega_2$, $\Omega_3$, and $\Omega_4$.
Here, $\Omega_1$ and $\Omega_4$ are cylindrical pipes of length $1~\text{m}$ and radius $0.05~\text{m}$.
The chambers $\Omega_2$ and $\Omega_3$ are the union of half a cylinder of radius $0.1~\text{m}$ and length $0.9~\text{m}$ and a $0.2\times 0.2\times 0.9~\text{m}^3$ square prism excluding $\Omega_1$ and $\Omega_4$, respectively.
Four-fifths of $\Omega_1$ and $\Omega_4$ are extended into $\Omega_2$ and $\Omega_3$.
Interface boundaries $\Gamma_{1}$, $\Gamma_{3}$, and $\Gamma_{5}$ are characterized by $\zeta_{1}= 0.01+2.5\text{i}$, $\zeta_{3}= -0.35\text{i}$, and $\zeta_{5} = 1.01-10.2\text{i}$, respectively, while at the open boundaries $\Gamma_{2}$ and $\Gamma_{4}$ we set $\zeta_{2}= \zeta_{4} =0$.
On the sound hard boundaries, $\Gamma_{\textrm{s}}$, the acoustic flux is zero, (that is, boundary condition~\eqref{Gammascond} holds). 

We model wave propagation in the muffler at wave number  $\kappa=26.6$ by boundary-value problem~\eqref{ContVarationalproblem} defined in the domain $\cup_{l=1}^4\Omega_l$. 
To impose an incoming wave of unit amplitude into $\Omega_1$ and to absorb outgoing planar waves from $\Omega_1$ and $\Omega_4$, we specify the values $g=1$, $g=0$  at $\Gamma_\text{io}\cap\partial\Omega_1$ and $\Gamma_\text{io}\cap\partial\Omega_4$, respectively. 
The interface conditions on $\Gamma_{i}$, for $i = 1,2,\dots,5$ are given by equation~\eqref{GammaIcond} using the corresponding $\zeta_{i}$.

 \begin{figure}
\centering
 \setlength{\unitlength}{0.5\textwidth}
\null\vspace{-20pt}
\subfloat{\includegraphics[width=90mm]{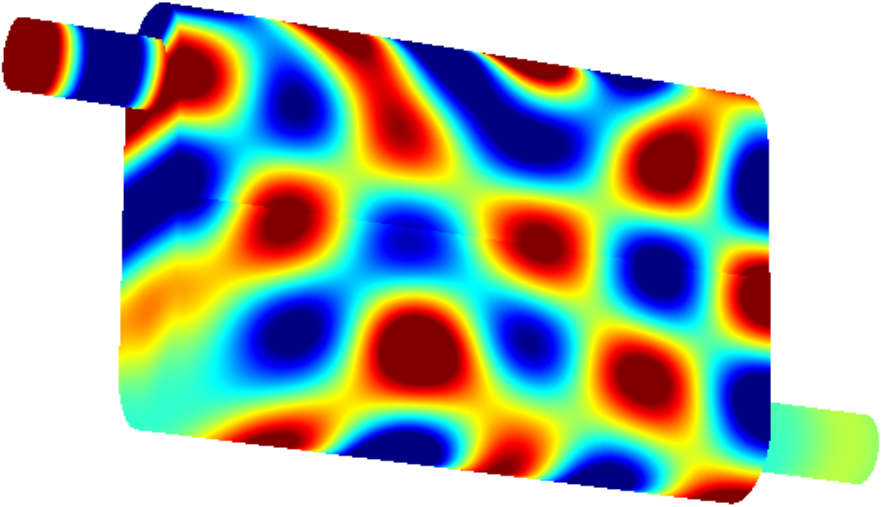}\label{fig:3DPressure1}}  
 \subfloat {\includegraphics[width=90mm]{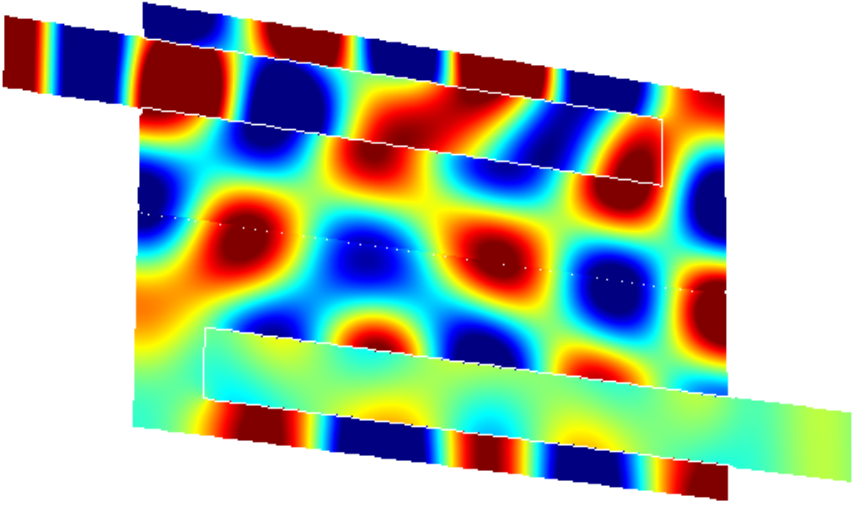}\label{fig:3DPressureSlice1}} 
 \caption{ 3D and sliced plots of the numerical solution of the 3D example with four different interface conditions.}
 \label{fig:3Dexample}
\end{figure}

Figure~\ref{fig:3Dexample} depicts the numerical solution obtained by our method, implemented in Comsol Multiphysics using the ``weak form'' facility.
The finite element discretization in $\Omega_i, \;i= 1,2,3,4$ uses second order tetrahedral elements on an unstructured mesh with maximum element size $h = 0.011~\text{m}$.   
The left picture in Figure~\ref{fig:3Dexample} show the real part of the pressure field on the outer boundaries, and the right image shows a sliced plot of the pressure at the plane $y=0$, where we can note the pressure jumps across the interface boundaries $\Gamma_{1}$, $\Gamma_{3}$, and $\Gamma_{5}$. 

\section*{Acknowledgements}

This research was supported in part by the Swedish Foundation for Strategic Research, Grant No.\ AM13-0029, and by the Swedish Research Council,  Grant No.\ 621-2013-3706. 

\begingroup
\raggedright
\sloppy
\bibliographystyle{model1-num-names}  
\bibliography{Nitschebib} 
\endgroup

\end{document}